\documentclass[12pt, reqno, oneside]{amsart}
\usepackage{macros}
\usepackage[
backend=biber,
url = false,
doi = false,
isbn = false
]{biblatex}
\appto{\bibsetup}{\raggedright}
\renewbibmacro{in:}{}
\numberwithin{equation}{section}    

\title{Nodal degeneration of chiral algebras I: Global structure and gluing formula}

\author{Elchanan Nafcha} 
\date{}


\begin{document}

\maketitle
\pagestyle{plain}

\begin{abstract}
    We define a natural extension of a universal factorization algebra $\cA$ to families of stable punctured curves, by integrating over all semistable modifications. We prove that the resulting sheaf of factorization homology satisfies a natural gluing formula, by tensoring over a certain derived associative algebra $\fZ_{\cA}^0$, generalizing the Verlinde formula for gluing of conformal blocks.
\end{abstract}

\tableofcontents

\section{Introduction}

\subsection{Motivation and main results}

The theory of factorization algebras arises when dealing with families of local operators over an algebraic variety or a smooth manifold. Its origin is the theory of vertex operator algebras --- a mathematical structure describing the chiral part of quantum field operators in a two-dimensional conformal field theory~\cite{Bor, FBZ}. In their original context, which is the subject of this paper, factorization algebras were introduced by Beilinson and Drinfeld~\cite{BD} for families of local operators over finite subsets of a smooth algebraic curve, arising in the study of the geometric Langlands correspondence. They were later generalized to higher dimensional smooth schemes~\cite{FG}, as well as smooth manifolds~\cite{HA, AF}, and have been used to describe the structure of observables in a general quantum field theory~\cite{CG}.

Heuristically, a factorization algebra over an algebraic curve $X$ is an assignment of a vector space (or a chain complex) $\cA_{(x_1,\dots,x_n)}$ for each finite subset $\{x_1,\dots,x_n\} \subset X$, and an isomorphism \[\cA_{(x_1,\dots,x_n)} \simeq \cA_{x_1} \otimes \dotsb \otimes \cA_{x_n}\] whenever $x_1,\dots,x_n$ are disjoint. The assignment \[(x_1,\dots,x_n) \mapsto \cA_{(x_1,\dots,x_n)}\] should have a flat connection along each variable $x_i$, and assemble to a collection of D-modules $\cA_I$ over $X^I$ for each finite set $I$, compatible with diagonal embeddings. A compatible collection of D-modules over powers of a curve $X$ can be equivalently defined as a D-module over $\Ran X$ --- the moduli space of finite subsets of $X$. The structure of a factorization algebra then translates to an equivariance condition with respect to the (disjoint) union operation. If $X$ is proper, one then has a de Rham global sections functor over $\Ran X$ --- the \textit{chiral (or factorization) homology of $X$ with coefficients in $\cA$}, which we denote by $\int_{X} \cA$.

Any vertex operator algebra $V$ defines a factorization algebra $\cA_{V,X}$ over any smooth curve~\cite{FBZ}. Factorization algebras arising in this way have a special property --- they are defined over any smooth curve in a compatible way, and assemble to a family of factorization algebras over the universal curve $\cX_g \to \cM_g$ of the moduli of genus-$g$ smooth curves. Forgetting the factorization structure, one gets the structure of a \textit{universal D-module}. The precise structure of the latter is defined in~\cite{Cliff}, using the stack of \'{e}tale germs of $n$-dimensional varieties. We give here a factorization version of this construction, which generalizes the definition in~\cite[Section~3.1.16]{BD} to the derived setting, using the moduli space $\MDisk_{\Ran}$ of  pointed multidisks:

\begin{definition}
    (Definition~\ref{def:multi_disks}, Definition~\ref{def:multi_disk_Ran}) Let $\MDisk_{\Ran}$ be the space classifying formal pointed multidisks, namely formal completions of finite flat divisors inside smooth curves. Let $\sqcup : \MDisk_{\Ran} \times \MDisk_{\Ran} \to \MDisk_{\Ran}$ be the disjoint union operation. Let $\Fact^{\univ}$ be the category of sheaves $\cA$ over $\MDisk_{\Ran}$, equivariant with respect to $\sqcup$, namely equipped with an isomorphism \[\sqcup^!\cA \iso \cA \boxtimes \cA\] We refer to objects of $\Fact^{\univ}$ as \textit{universal factorization algebras}.
\end{definition}

\begin{theorem}
    (Theorem~\ref{thm:loc-fact}) For any family of smooth curves $X/S$, there exists a classifying map \[\pi^{X/S}_{\dR} : \Ran (X/S) \to \MDisk_{\Ran}\] such that the pullback map $\pi_{\dR}^{X/S,!}$ sends a universal factorization algebra to an $S$-family of factorization algebras $\cA_{X/S}$ over $X/S$. 
\end{theorem}

In the case where a universal factorization algebra arises from a vertex operator algebra, this construction reduces to that of~\cite{FBZ}. In particular, we have an isomorphism $H_0\int_{X} \cA_X \simeq \VV(V)_{[X]}$, where $\VV(V)_{[X]}$ denotes the space of coinvariants (with vacuum coefficients), dual to the space of conformal blocks, which is central in the construction of correlation functions in conformal field theory.  

If we start with a family of curves $X/S$, the resulting chiral homology would be a sheaf over $S$. Furthermore, using the action of the Virasoro algebra, we get a structure of a twisted D-module~\cite{BS}. In particular, taking $S$ to be the universal base $\cM_{g}$, one gets a twisted D-module of (complexes of) vector spaces over $\cM_g$, whose zeroth homology is the sheaf of coinvariants. An important question in the study of vertex operator algebras is the extension of the sheaf of coinvariants to the boundary, namely to the Deligne-Mumford compactification $\bcM_g$, together with a gluing formula, relating the space of coinvariants for a nodal curve with that of its normalization. The motivation is twofold:

The first motivation is computational: Under certain finiteness assumptions, the sheaf of coinvariants is a vector bundle, and in particular, its rank depends on the genus only. By extending to the boundary, one reduces the computation to the case of a nodal curve. The gluing formula then expresses the genus-$g$ rank in terms of lower genera. This is known as the \textit{Verlinde formula}, which was conjectured in~\cite{Ver}, and proven in~\cite{TUY} for the case of the integrable quotient vertex algebra, and more generally in~\cite{DGT}.

The second comes from the idea of an algebraic quantum field theory~\cite{BFM, BZSV}. In its simplest form, it describes a structure of a topological quantum field theory arising from a conformal field theory which is equipped with strong equivariance with respect to deformations. In that case, the resulting structure is that of a functor out of a $[1,2]$-cobordism category~\cite{Seg}. In the case of a vertex operator algebra $V$, it assigns the category $V\mh\Mod$ to the circle, and the functor \[\VV(V,-)_{[\overline{\Sigma},\{x_i\}]} : V\mh\Mod^{\otimes n} \to \Vect\] to a Riemann surface $\Sigma$ with boundary $(S^1)^{\sqcup n}$, where $\VV(V,{M_1,\dots,M_n})_{[\overline{\Sigma},\{x_i\}]}$ denotes the spaces of coinvariants corresponding to $M_i$ assigned to $x_i$. Here $\overline{\Sigma} = \Sigma \underset{(S^1)^{\sqcup n}}{\cup} D^{\sqcup n}$, and $x_i$ the center of the $i$-th disk. The gluing law then translates to a composition law --- the behavior of the spaces $\VV(V,\{M_i\})_{[\overline{\Sigma},\{x_i\}]}$ with respect to compositions of cobordisms. The precise structure is described in~\cite{DW} under strong finiteness assumptions on the category of modules. In general, however, one cannot expect a completely topological structure: even if we have a flat connection on $\cM_g$, this connection may fail to be integrable. Instead, one only expects a composition law to hold for nodal curves.

Our main result is an extension of chiral homology to the boundary, and a gluing formula for nodal curves. 
\begin{theorem}
    (Theorem~\ref{thm:disj-gluing}, Theorem~\ref{thm:self-gluing}) For a universal factorization algebra $\cA$ and integers $g,n \geq 0$ such that $2g + n - 2 > 0$, there exists a sheaf $\int_{g,n} \cA$ over $\bcM_{g,n}$ extending the sheaf of chiral homology over $\cM_g$, and a derived non-unital associative algebra $\fZ_{\cA}^0$, such that the fibers $\int_{X \backslash \{x_1,\dots,x_n\}} \cA$ of $\int_{g,n}\cA$ admit $n$ commuting $\fZ_{\cA}^0$ actions, and such that the following gluing formulas hold:
    \begin{enumerate}
        \item If a curve is given by the gluing of two curves $(X,x)$, $(Y,y)$, we have \begin{equation} \label{eq:gluing-formula-1}
            \int_{X \underset{x \sim y}{\cup} Y} \cA \simeq \int_{X \backslash \{x\}} \cA \underset{\fZ_{\cA}^0}{\otimes} \int_{Y \backslash \{y\}} \cA~.
        \end{equation}
        \item If a curve is given by self-gluing of a two-pointed curve $(X,x_1,x_2)$, we have \begin{equation} \label{eq:gluing-formula-2}
            \int_{X \underset{\{x_1, x_2\}}{\cup} \{\pt\}} \cA \simeq \int_{X \backslash \{x_1,x_2\}} \cA \underset{\fZ_{\cA}^0 \otimes \fZ_{\cA}^{0,\op}}{\otimes} \fZ_{\cA}^0~.
        \end{equation} In other words, it is given by the Hochschild homology of $\fZ_{\cA}^0$ with coefficients in the bimodule $\int_{X \backslash \{x_1,x_2\}} \cA$. 
    \end{enumerate}
\end{theorem}

Here the tensor product of two modules $M,N$ over a non-unital associative algebra $A$ is given by the non-unital bar complex \[M \underset{A}{\otimes} N = \colim(\dotsb \rrrarrow M \otimes A \otimes N \rrarrow M \otimes N)~.\]

The fibers $\int_{X \backslash \{x_1,\dots,x_n\}} \cA$ are isomorphic to the chiral homology of the proper curve $X$ with coefficients in a certain $\cA$-module $\fZ_{\cA}^+$ supported at each marked point $x_i$, which we regard as a puncture. Similarly, the fibers at a nodal curve $X$ can be interpreted as factorization homology for $X$ with coefficients in a chiral bimodule $\fZ_{\cA}$. However, the latter description does not work in families --- in a nodal degeneration with smooth general fiber, the node does not extend to a section over the smooth locus. Instead, we interpret it as chiral homology for a module $\Vac(\cA)^{\epsilon}_{(X,x_1,\dots,x_n)}$ whose support is the entire space $\Ran X$.

In the second part~\cite{Naf}, we describe the local structure of the puncture modules $\fZ_{\cA}^+$ and $\fZ_{\cA}^-$, as well as the node bimodule $\fZ_{\cA}$, and relate them to more familiar objects from the theory of vertex algebras. Specifically, we show that for a vertex operator algebra $V$ with associated universal factorization algebra $\cA_V$, the associative algebra $H_0\fZ_{\cA}^0$ is isomorphic to Zhu's associative algebra $A(V)$ defined in~\cite{Zhu}, and the chiral module $H_0\fZ_{\cA}^+$ is the vertex algebra module associated to $A(V)$. Assume now that $A(V)$ is semisimple, so that it decomposes as $A(V) \simeq \bigoplus_{i=1}^{\ell} M_i \otimes M_i^{\lor}$, where $M_1,\dots,M_{\ell}$ are the simple $A(V)$-modules. In particular, $H_0\fZ_{\cA}^+$ decomposes as $H_0\fZ_{\cA}^+ \simeq \bigoplus_{i=1}^{\ell} \cM_i \otimes \cM_i^{\lor}$, where $\cM_i$ is the irreducible $V$-module associated to $M_i$. Then we show that at the level of $H_0$, the formula~\eqref{eq:gluing-formula-1} reduces to \[\VV(V)_{[X \underset{x \sim y}{\cup} Y]} \simeq \bigoplus_{i=1}^{\ell} \VV(V, \cM_i)_{[X,x]} \otimes \VV(V, \cM_i^{\lor})_{[Y,y]}\] This is the usual form of the Verlinde formula, which is proven in~\cite{TUY,DGT}.

An important ingredient, which is missing from the current paper, is that of \textit{sewing}, namely the extension of a gluing formula to the formal neighborhood of the boundary $\partial \bcM_{g}$. In~\cite{Naf} we identify the formal deformation $H_0\tilde{\fZ}_{\cA}$ of $H_0\fZ_{\cA}$ with a certain quotient of a trivial deformation of the universal enveloping algebra. In order to obtain a sewing formula, one needs to prove the triviality of this deformation, and extend it to triviality of the derived deformation $\tilde{\fZ}_{\cA}$ of $\fZ_{\cA}$. 

\subsection{Overview}

In Section~\ref{sec:ch}, we review the definitions of chiral and factorization algebras, chiral homology and chiral modules, and write down the corresponding definitions in the relative case of families of smooth curves. We formalize the notion of a universal factorization algebra as a certain commutative coalgebra in some symmetric monoidal category, and define a localization functor from the category of universal factorization algebras to the category of families of factorization algebras over any family of smooth curves.

In Section~\ref{sec:moduli}, we define the moduli space of semistable modifications of a curve $\fM_{X}^{\sst}$, and construct various Ran spaces (namely, moduli spaces of finite subsets) over it. In particular, we define the pseudo-proper prestack $\bcM_{g,n,\Ran} \to \fM^{\sst}_{g,n}$ of stable configurations over semistable modifications, which serves as our maximal compactification for the Ran space of the punctured smooth locus $\mathring{\Ran}_{g,n}$, and which maps to the minimal compactification $\Ran_{g,n}$, given by Ran space of the universal curve.

In Section~\ref{sec:chial-nodal}, we define chiral monoidal structures for the spaces of semistable modifications, as well as a chiral module space structure on the space $\bcM_{g,n,\Ran}$ of stable configurations. We then construct, for any universal factorization algebra, a factorization $\cA$-module $\Vac(\cA)_{g,n}^{\epsilon}$ over $\bcM_{g,n,\Ran}$. By pushing forward along $\bcM_{g,n,\Ran} \to \Ran_{g,n}$, we get a factorization $\cA$-module $\Vac(\cA)^0_{g,n}$ over $\Ran_{g,n}$, whose fibers over unmarked smooth points are the vacuum module. We prove that in the case of two-pointed genus zero curves $\bcM_{0,2,\Ran}$, we have an additional (non-symmetric) monoidal structure, which we call the join monoidal structure, and which is given by concatenation of semistable modifications, and show that $\Vac(\cA)_{g,n}^{\epsilon}$ is equivariant with respect to this monoidal structure.

Finally, in Section~\ref{sec:ch-hom}, we prove the main result of this paper. We define the associative algebra $\fZ_{\cA}^0$ for any universal factorization algebra $\cA$ as global sections over $\bcM_{0,2,\Ran}$, with multiplication induced by the join product. We also define $\fZ_{\cA}^0$-modules $\fZ_{\cA}^+,\fZ_{\cA}^-$, which admit an additional structure of chiral modules supported at a point. We define the chiral bimodule $\fZ_{\cA}$, which is the value of our extension at singular points. We then prove a gluing formula, relating factorization homology with coefficients in $\fZ_{\cA}$ at a node to factorization homology with coefficients in $\fZ_{\cA}^{\pm}$ over its normalization, and in particular proving that $\fZ_{\cA} \simeq \fZ_{\cA}^+ \underset{\fZ_{\cA}^0}{\otimes} \fZ_{\cA}^-$.

\subsection{Notations and conventions}

Throughout this paper, and unless specified otherwise, all categories are assumed to be $(\infty,1)$-categories. In particular, for a field $k$, $\Vect_k$ will refer to the $(\infty,1)$-category of $k$-chain complexes. We use the theory of ind-coherent sheaves as developed in~\cite{Gai-ICS}. We use the following notations:
\begin{itemize}
    \item $k$ denotes a fixed algebraically closed field of characteristic zero.
    \item For a base scheme or stack $S$, $\PreSt_S$ refers to the category of $S$-prestacks, namely presheaves of groupoids over the category of affine schemes over $S$. We will write $\PreSt = \PreSt_k$.
    \item $\PreSt^{\aft} \subset \PreSt$ denotes the full subcategory of prestacks \textit{locally almost of finite type}.
    \item For a prestack $\cX$, we write $\cX_{\dR}$ for its de Rham prestack, whose functor of points is given by $\cX_{\dR}(R) = \cX(R^{\red})$. For a base stack $S$ and $\cX \in \PreSt_S$, we write \[\cX_{S\mh\dR} = \cX_{\dR} \times_{S_{\dR}} S\] the relative de Rham prestack. If the base stack is understood implicitly, such as in the case $S = \fM_{g,n}^{\sst}$, we will write \[\cX_{S\mh\dR} = \cX_{\dR/}\]
    \item For a category $\cC$, we let $\cC^{\corr}$ be the $1$-category of correspondences in $\cC$, namely the category with the same objects as in $\cC$, and with morphisms $X \to Y$ given by spans $X \from Z \to Y$ in $\cC$. Compositions are defined via pullback.
    \item We will denote by $\Cat_k = \Vect_k\mh\Mod(\Cat)$ the symmetric monoidal category of $k$-linear categories, namely stable categories whose hom spaces have a structure of $k$-vector spaces in a compatible way. The monoidal structure is given by Lurie's tensor product~\cite[Section~4.8]{HA}.
    \item We denote by $\hat{\cO} = k[[t]]$ the algebra of Taylor series over $k$, and by $\hat{D} = \Spf k[[t]]$ its formal spectrum. We denote by \[\Aut(\hat{\cO}) \simeq k[[a_0]][a_1^{\pm 1},a_2,a_3,\dots]\] the group indscheme of continuous automorphisms of $\hat{\cO}$, and \[\Aut_*(\hat{\cO}) \simeq k[a_1^{\pm 1},a_2,a_3,\dots]\] its subgroup of automorphisms preserving the closed point $\Spec k \to \Spf k[[t]]$.
    \item By a flat family of curves $X \to S$ over a base stack $S$ we mean a flat morphism, of finite presentation, of relative dimension $\leq 1$, with $X$ an algebraic space. 
    \item The category $\fSet^{\surj}$ is the category of non-empty finite sets and surjections between them. The category $\fSet^{\surj,\subset}$ refers to the category of non-empty finite sets, together with a choice of a non-empty subset $I \subset J$. Morphisms are given by surjective morphisms of sets which induce surjective morphisms on the chosen subsets.
    \item The category $\Delta$ is the simplex category of finite non-empty ordered sets $[n] = \{0 < 1 < \dotsb < n\}$ and order-preserving morphisms between them. The subcategories $\Delta^{\surj},\Delta^{\inj}$ are the subcategories with the same objects, and with morphisms given by surjective and injective maps resp.  
\end{itemize}

For the most part, we will only deal with the subcategories $\PreSt^{\aft}_S$ over an algebraic stack $S$ locally of finite type. We will use the symmetric monoidal structure on $\PreSt_S^{\corr}$ induced by the Cartesian symmetric monoidal structure $(-) \times_S (-)$ on $\PreSt_S$. We then have a lax-monoidal functor \[\IndCoh : \PreSt^{\aft,\corr} \to \Cat_k\] sending a prestack $\cX$ to the category \[\IndCoh(\cX) = \lim_{\Spec R \to \cX} \IndCoh(\Spec R)\] where the limit is taken with respect to $!$-pullback, and a morphism given by a span $\cX \xleftarrow{f} \cZ \xrightarrow{g} \cY$ is sent to the functor \[g_*f^! : \IndCoh(\cX) \to \IndCoh(\cY)\] The lax-monoidal structure is given by the exterior tensor product \[\IndCoh(\cX) \otimes \IndCoh(\cY) \xrightarrow{\boxtimes} \IndCoh(\cX  \times \cY)~.\]

As a consequence, for any base stack $S \in \PreSt^{\aft}$, the category $\IndCoh(S)$ is a commutative algebra in $\Cat_k$, with multiplication given by \[\IndCoh(S) \otimes \IndCoh(S) \xrightarrow{\boxtimes} \IndCoh(S \times S) \xrightarrow{\Delta^!} \IndCoh(S)\] where $\Delta : S \to S \times S$ is the diagonal map. We let \[\Cat_S \coloneqq \IndCoh(S)\mh\Mod(\Cat_k)~.\]

The composition \[\IndCoh_S : \PreSt_S^{\aft,\corr} \to \PreSt^{\aft,\corr} \to \Cat_k\] then upgrades to a lax-monoidal functor \[\IndCoh_S : \PreSt_S^{\aft,\corr} \to \Cat_S\] where the lax structure is given by the composition \begin{align*}
    \IndCoh(\cX) \underset{\IndCoh(S)}{\otimes} \IndCoh(\cY) & \simeq \IndCoh(S) \underset{\IndCoh(S) \otimes \IndCoh(S)}{\otimes} \IndCoh(\cX) \otimes \IndCoh(\cY) \\
    & \xrightarrow{\id \otimes \boxtimes} \IndCoh(S) \underset{\IndCoh(S \times S)}{\otimes} \IndCoh(\cX \times \cY) \\
    & \to \IndCoh(S \times_{S \times S} \cX \times \cY) \simeq \IndCoh(\cX \times_S \cY)~.
\end{align*} 
Denote this composition by \[\otimes_S^! : \IndCoh(\cX) \otimes_{\IndCoh(S)} \IndCoh(\cY) \to \IndCoh(\cX \times_S \cY)~.\]

\subsection{Ind-coherent sheaves over infinite type stacks}

Since the group indscheme $\Aut(\hat{\cO})$ has infinite type, we will also need some extensions of the theory of ind-coherent sheaves to the infinite context. We will follow the formalism in~\cite{RasHM,CF}. Unlike the finite type case, in the infinite case we generally do not have both $!$-pullback and $*$-pushforward. 

One can always define a functor \[\IndCoh^! : \PreSt^{\op} \to \Cat_k\] by right Kan extension, so that we have $!$-pullback along any map $f : \cX \to \cY$. However, not much can be said about the category $\IndCoh^!(\cX)$ without further assumptions. One can also define a $*$-theory for a larger class of prestacks:

\begin{definition} \label{def:apai}
    (\parencite[Section~11.2]{CF}) A qcqs, eventually coconnective scheme is called \textit{apaisant} if it can be written as an inverse limit under faithfully flat maps of eventually coconnective schemes almost of finite type. A morphism $f : \cX \to \cY$ between prestacks is relatively apaisant if it is schematic, qcqs, has bounded Tor dimension, and its base change to any apaisant scheme is apaisant.
\end{definition}

\begin{definition} \label{def:weakly-ren}
    (\parencite[Definition~6.21.1]{RasHM}, \parencite[Definition~11.5.1]{CF}) A prestack $S$ is \textit{weakly renormalizable} if it admits a flat and relatively apaisant cover $T \to S$ with $T$ an apaisant indscheme, namely $T$ can be written as a filtered colimit under closed embeddings of apaisant schemes. We let $\PreSt_{\operatorname{w.ren}}$ be the full subcategory spanned by weakly renormalizable prestacks.
\end{definition}

For any renormalizable prestack $\cX$, one can attach a compactly generated category $\IndCoh_*^{\ren}(\cX)$, and this construction defines a functor \[\IndCoh_*^{\ren} : \PreSt_{\operatorname{w.ren}} \to \Cat_k~.\]

\begin{definition} \label{def:IndCoh_ren}
    Define \[\IndCoh^!_{\ren} : \PreSt_{\operatorname{w.ren}}^{\op} \to \Cat_k\] to be the dual of $\IndCoh_*^{\ren}$.
\end{definition}

We will need the following properties of this functor:
\begin{lemma} \label{lem:ren-?push}
    (Cf.~\cite[Section~6.24]{RasHM}, \cite[Section~11.9]{CF}) For an open immersion $f : \cX \to \cY$ in $\PreSt_{\operatorname{w.ren}}$, the functor $f_{\ren}^!$ admits a continuous right adjoint $f^{\ren}_?$, such that for any pullback diagram 
    \[\begin{tikzcd}
        {\cX'} & {\cY'} \\
        {\cX} & {\cY}
        \arrow["{f'}"', from=1-1, to=1-2]
        \arrow["{g'}", from=1-1, to=2-1]
        \arrow["g"', from=1-2, to=2-2]
        \arrow["f", from=2-1, to=2-2]
    \end{tikzcd}\] in $\PreSt_{\operatorname{w.ren}}$, the Beck-Chevalley map \[g'^{!}_{\ren}f'^{\ren}_{?} \to g_{\ren}^!f_?^{\ren}\] is an isomorphism.
\end{lemma}

\begin{lemma} \label{lem:ren-!push}
    (Cf.~\cite[Section~11.9]{CF}) For a relatively apaisant and ind-proper ind-finitely presented $f : \cX \to \cY$ in $\PreSt_{\operatorname{w.ren}}$, the functor $f_{\ren}^!$ admits a left adjoint $f^{\ren}_!$, and for any pullback diagram 
    \[\begin{tikzcd}
        {\cX'} & {\cY'} \\
        {\cX} & {\cY}
        \arrow["{f'}"', from=1-1, to=1-2]
        \arrow["{g'}", from=1-1, to=2-1]
        \arrow["g"', from=1-2, to=2-2]
        \arrow["f", from=2-1, to=2-2]
    \end{tikzcd}\] in $\PreSt_{\operatorname{w.ren}}$, the Beck-Chevalley map \[f'^{\ren}_!g'^!_{\ren} \to g_{\ren}^!f_{\ren,!}\] is an isomorphism.
\end{lemma}

\subsection{Relation to previous work}

For a vertex operator algebra $V$ and a curve $X$, Frenkel and Ben-Zvi~\cite{FBZ} constructed a family of D-modules $\cV_{g}$ over the universal curve $\cX_{g}$, using Gelfand-Kazhdan descent. This can be interpreted as a classifying map $X \to B\Aut_*\hat{\cO}$ for any family of smooth curves $X/S$, which agrees with our map $\pi^{X/S}_{1}$ from Example~\ref{ex:X-to-disk}, and our construction reduces to their construction in the case where a universal factorization algebra comes from a vertex operator algebra.

Our definition of the space of $I$-pointed multidisks is borrowed from~\cite{Lur}, and provides a (one-dimensional) multi-point version of the stack of \'{e}tale germs defined in~\cite{Cliff}.

In~\cite{DG}, the sheaf $\cV_g$ is extended to a sheaf $\overline{\cV}_g$ over the universal curve $\bcX_g \to \bcM_{g}$, using the isomorphism $\bcX_{g} \simeq \bcM_{g,1}$. This allows one to identify a node $p \in X \subset \bcX_g$ with a smooth point in the stable model for $(X,p)$, given by $1 \in (X \underset{\{p\} \sim \{0,\infty\}}{\cup} \PP^1) / \GG_m \subset \bcM_{g,1}$. Since the vertex algebra bundle can be defined over any smooth point, one gets the desired extension. Inspired by this construction, we define the extension of a universal factorization algebra using the moduli of semistable modifications at $p$, with a sequence of curves $\PP^1 \cup \dotsb \cup \PP^1$ of arbitrary length inserted at $p$, and with an arbitrary number of smooth points over it. This provides a multi-point extension for the above construction. However, we do not know whether the resulting spaces of conformal blocks agree with $H_0$ of our extension, beyond the rational case. In~\cite{DGK}, the above results are generalized using the mode-transition algebra, which we identify with the zeroth homology of $\fZ_{\cA}$ in~\cite{Naf}.

A related approach is developed in~\cite{VM}, where the author constructs rigid analytic factorization algebras as functors out of categories of open sets, and proves a corresponding gluing result along the rigid analytic generic fiber of a nodal degeneration family. We do not know the precise relation between the two approaches.

\subsection{Acknowledgements}

I am deeply grateful to my advisor, John Francis, for his invaluable guidance and support throughout the writing of this paper. I would like to thank Alexander Beilinson for generously sharing his insights on the problem and for his meticulous feedback, which helped correct several errors in my arguments. I also thank Sam Raskin, David Nadler, Chiara Damiolini, Kevin Lin, and Justin Campbell for many helpful discussions.

\section{Chiral and factorization algebras} \label{sec:ch}

\subsection{Chiral monoidal structure}

A factorization algebra $\cA_X$ over a smooth curve $X$ is given roughly by the following data: a collection of D-modules $\cA_{I}$ over $X^I$ for each finite set $I$, compatible with restrictions along diagonal embeddings $X^I \to X^J$, together with a collection of isomorphisms \[\cA_I \mid X^I \backslash \Delta \simeq \cA_1^{\boxtimes I} \mid X^I \backslash \Delta\] where $\Delta \subset X^I$ is the union of all the diagonal subspaces. In particular, both the values over the diagonal and away from the diagonal are determined by $\cA_1$, and the only missing data are the gluing rules between $\cA_1^{\boxtimes I}$ over $X^I \backslash \Delta$ and $\cA_1$ over $\Delta$. The latter is captured precisely by the structure of a chiral algebra: A D-module $\cB$ over $X$, and a collection of gluing rules \[\cB^{\boxtimes I} \mid X^I \backslash \Delta \to \Delta_!\cB~.\] These gluing rules are then expected to satisfy certain compatibility conditions. In~\cite{BD}, Beilinson and Drinfeld formulated the structures of chiral and factorization algebras as certain cocommutative and Lie algebras in the appropriate pseudo-tensor category (namely, a $k$-linear colored operad), and proved that they are in fact equivalent. This was later reinterpreted in~\cite{FG} as Koszul duality between Lie and cocommutative algebras in a certain symmetric monoidal category. One possible way of implementing this symmetric monoidal category is via the Ran space --- the moduli space of finite nonempty subsets of $X$. We will describe this idea, in the more general context of families of smooth curves, following the formalism in~\cite{FG}.

Throughout this section, we fix an algebraic stack of finite type $S$. We will only discuss here the case of curves.

\begin{definition}
    Given a family of smooth curves $X \to S$, define a prestack $\Ran (X/S) \to S \in \PreSt_S^{\aft}$ to be the functor assigning to a map $f : \Spec R \to S$ the set
    \[\Ran (X/S)(f) = \left\{\text{finite nonempty sets of lifts }\Spec R^{\red} \to X\right\}\]
\end{definition}

Let $\fSet^{\surj}$ be the category of finite sets and surjective morphisms between them. Define a functor
\[(X/S)^{\fSet}_{\dR} : \fSet^{\op} \to \PreSt_{S}^{\aft}\]
by
\[I \mapsto X^{\times_S I}_{S\mh\dR}, \quad (I \xrightarrow{\alpha} J) \mapsto (X^{\times_S J}_{S\mh\dR} \xrightarrow{\Delta_{(\alpha)}} X^{\times_S I}_{S\mh\dR})~,\]
where $\Delta_{(\alpha)}$ is the diagonal map corresponding to $\alpha$. Then we have a presentation of $\Ran (X/S)$ as a colimit:

\begin{lemma} \label{lem:Ran-colim}
    $\Ran(X/S) \simeq \colim_{\fSet^{\surj,\op}} X^{\times_S I}_{S\mh\dR}$.
\end{lemma}

The $S$-prestack $\Ran (X/S)$ admits a symmetric monoidal structure  given by the union of finite subsets. If we restrict the union operation to disjoint subsets only, we get merely an operad structure, which can be implemented by a symmetric monoidal structure in the category $\PreSt^{\corr}_S$ of correspondences in $\PreSt_S$, as we describe below.

\begin{definition}
    Let 
    \[\jmath : (\Ran (X/S) \times_S \Ran (X/S))_{\disj} \subset \Ran (X/S) \times_S \Ran (X/S)\]
    be the sub-prestack whose $\Spec R \to S$-points are given by the subset of lifts $\Spec R \to X^{\times_S I}_{S\mh\dR} \times_S X^{\times_S J}_{S\mh\dR}$ whose graph factors through the complement to the diagonal $\Delta \subset X \times_S X$.

    Define the chiral symmetric monoidal structure to be the correspondence
    \[\Ran (X/S) \times_S \Ran (X/S) \xleftarrow{\jmath} (\Ran (X/S) \times_S \Ran (X/S))_{\disj} \xrightarrow{\sqcup} \Ran (X/S)~.\]
\end{definition}

\begin{definition}
    Let $\IndCoh^{\ch}(\Ran (X/S))$ be the symmetric monoidal structure on the category $\IndCoh(\Ran (X/S))$ induced by the disjoint union operation, namely
    \[\otimes^{\ch}_{S,i\in I} \cF_i = \sqcup_*\jmath^!\otimes^!_{S,i \in I} \cF_i~.\]
\end{definition}

From Lemma~\ref{lem:Ran-colim}, we get
\[\IndCoh(\Ran (X/S)) \simeq \lim_{I \in \fSet^{\surj}, (-)^!} \IndCoh(X^{\times_S I}_{S\mh\dR})~.\]

Furthermore, since all the transition maps $\Delta_{(\alpha)}$ are closed embeddings, the corresponding pullbacks $\Delta_{(\alpha)}^!$ admit left adjoints $\Delta_{(\alpha),!} \simeq \Delta_{(\alpha),*}$. In particular, we get an isomorphism 
\begin{equation} \label{eq:colim-present}
    \cF \simeq \colim_{\fSet^{\op}} \Delta_{I,!}\Delta_{I}^!\cF \eqqcolon \colim_{\fSet^{\op}}\Delta_{I,!}\cF_{I}
\end{equation}
for any $\cF \in \IndCoh(\Ran (X/S))$. Here $\Delta_I : X_{S\mh\dR}^{\times_S I} \to \Ran (X/S)$ is the diagonal embedding.

In terms of this presentation, an ind-coherent sheaf $\cF$ over $\Ran (X/S)$ can be described as a collection $\cF_I \in \IndCoh(X_{S\mh\dR}^{\times_S I})$ for $I \in \fSet$, 
together with isomorphisms
\begin{equation} \label{eq:Ran-iso}
    \cF_I \iso \Delta_{(\alpha)}^!\cF_J
\end{equation} for any surjective $\alpha : J \onto I$. The chiral monoidal structure can then be described by \[(\otimes^{\ch}_{S,i \in I} \cF_i)_J = \bigoplus_{\alpha : J \onto I} \jmath_{(\alpha),*}\jmath_{(\alpha)}^!\otimes^!_{S,i \in I}(\cF_i)_{J_i}~.\]

Here $\jmath_{(\alpha)} : X^{(\alpha)}_{S\mh-\dR} \into X^{\times_S I}_{S\mh\dR}$, where \[X^{(\alpha)}_{S\mh\dR}(R) = \left\{(x_i)_{i \in I} \in X^{\times_SI}_{S\mh\dR}(R) :\: x_i \cap x_{i'} = \emptyset \text{ for }\alpha(i) \neq \alpha(i')\right\}~.\]

Now, since $\sqcup$ is \'{e}tale (see, e.g., \cite[Lemma~6.18.1]{Ras}), $\sqcup^! = \sqcup^*$ is left adjoint to $\sqcup_*$.
\begin{definition}
    A family of factorization algebras over $X/S$ is a commutative coalgebra $\cA \in \coComm(\IndCoh^{\ch}(\Ran(X/S)))$ such that for each finite set $I$, the morphism \begin{equation} \label{eq:fact-iso}
        \sqcup^!\cA \to \jmath^!\cA^{\otimes_S I}
    \end{equation} adjoint to the comultiplication map, is an isomorphism. We denote by $\Fact(X/S)$ the category of families of factorization algebras over $X/S$. 
\end{definition}

Given a map $f: T \to S$, we have a symmetric monoidal base-change functor \[(-)_T : \PreSt_S \to \PreSt_T\] right adjoint to the forgetful functor \[\Oblv_S^T : \PreSt_T \to \PreSt_S\] It induces a symmetric monoidal functor \[(-)_T : \PreSt_S^{\corr} \to \PreSt_T^{\corr}\] and therefore an induced map on commutative (co)algebra objects. Given a smooth family of curves $X/S$, it sends the commutative algebra $\Ran(X/S)$ to a commutative algebra $\Ran(X/S)_T$. In general, this commutative algebra is different from the commutative algebra $\Ran(X_T/T)$: While the underlying spaces are the same, the notions of disjointness are different, namely $\Ran(X/S)^{\times_SI}_{\disj} \times_S T$ is different from $\Ran(X_T/T)^{\times_TI}$. We always have a map 
\begin{equation} \label{eq:bc-Ran-spaces}
    \Ran(X/S)^{\times_SI}_{\disj} \times_S T \to \Ran(X_T/T)^{\times_TI}_{\disj}~, 
\end{equation} 
but this map may not be an isomorphism even if $u_f$ is \'{e}tale. 

\begin{definition} \label{def:Ran-compat-bc}
    We say the chiral monoidal structure on $\Ran(X/S)$ is compatible with base-change along $f : T \to S$ if the map \eqref{eq:bc-Ran-spaces} is an isomorphism for each finite set $I$.
\end{definition}

\begin{example}
    If $u_f$ is an open embedding, $\Ran(X/S)$ is compatible with base-change along $f$. 
\end{example}

We have a unit map \[u_f : \Oblv^T_S\Ran(X/S)_T \to \Ran(X/S)\]
in $\PreSt_S$, which induces maps \[(\id : u_f) : \Ran(X/S)_T \xleftarrow{\id} \Ran(X/S)_T \xrightarrow{u_f} \Ran(X/S)\] and \[(u_f : \id) : \Ran(X/S) \xleftarrow{u_f} \Ran(X/S)_T \xrightarrow{\id} \Ran(X/S)\] in $\PreSt_S^{\corr}$.

\begin{lemma} \label{lem:base-change-fact}
    Given $f :T \to S$, and a family of smooth curves $X / S$, such that $\Ran(X/S)$ is compatible with base-change along $f$, both $(\id : u_f) : \Ran(X_T/T) \to \Ran(X/S)$ and $(u_f : \id) : \Ran(X/S) \to \Ran(X_T/T)$ are maps of commutative algebras. Furthermore, we have induced symmetric monoidal maps \[u_{f,*} : \IndCoh^{\ch}(\Ran(X_T/T)) \leftrightarrows \IndCoh^{\ch}(\Ran(X/S)) : u_f^!\] and corresponding functors \[u_{f,*}^{\fact} : \Fact(X_T/T) \to \Fact(X/S) : u_f^{!,\fact}~.\]
\end{lemma}

\begin{proof}
    First, we need to show that $u_f$ commutes both with the multiplication maps \[\Ran(X_T/T)^{\times_TI} \leftarrow{\jmath_T} \Ran(X_T/T)^{\times_TI}_{\disj} \xrightarrow{\sqcup_T} \Ran(X_T/T)\] and with the comultiplication \[\Ran(X_T/T) \leftarrow{\sqcup_T} \Ran(X_T/T)^{\times_TI}_{\disj} \xrightarrow{\jmath_T} \Ran(X_T/T)^{\times_TI}\] Since the composition in correspondence categories is given by pullback, this amounts to proving that we have commutative diagrams  \[\begin{tikzcd}
	{\Ran(X_T/T)^{\times_T I}} & {\Ran(X_T/T)^{\times_T I}_{\disj}} & {\Ran(X_T/T)} \\
	{\Ran(X/S)^{\times_SI}} & {\Ran(X/S)^{\times_SI}_{\disj}} & {\Ran(X/S)}
	\arrow["{u_f^I}", from=1-1, to=2-1]
	\arrow["{\jmath_T}"', from=1-2, to=1-1]
	\arrow["{\sqcup_T}", from=1-2, to=1-3]
	\arrow["{u_{f,I}}", from=1-2, to=2-2]
	\arrow["{u_f}", from=1-3, to=2-3]
	\arrow["\jmath"', from=2-2, to=2-1]
	\arrow["\sqcup", from=2-2, to=2-3]
    \end{tikzcd}\]
    where both squares are Cartesian. This follows from compatibility with pullback. Since $\IndCoh$ is lax-monoidal, it sends maps of commutative algebras to maps of commutative algebras, and thus induces symmetric monoidal functors \[u_{f,*} : \IndCoh^{\ch}(\Ran(X_T/T)) \leftrightarrows \IndCoh^{\ch}(\Ran(X/S)) : u_f^!\] In particular, both functors send a commutative coalgebra to a commutative coalgebra, and it remains only to check the factorization isomorphism~\eqref{eq:fact-iso}. Given $\cA \in \Fact(X/S)$, the maps \[\sqcup^!_Tu_f^!\cA \simeq u_{f,I}^!\sqcup^!\cA \to u_{f,I}^!\jmath^!\cA^{\otimes_S^!I} \simeq \jmath^!_T (u_f^!\cA)^{\otimes_T^!I}\] are given by pullback along $u_{f,I}$ of the factorization isomorphisms for $\cA$, and therefore are isomorphisms as well. On the other hand, given $\cA \in \Fact(X_T/T)$, by using base-change for both squares we have that \[\sqcup^!u_{f,*}\cA \simeq u_{f,I,*}\sqcup_T^!\cA \to u_{f,I,*}\jmath_T^!\cA^{\otimes_TI} \simeq \jmath^!(u_{f,*}\cA)^{\otimes_SI}\] are isomorphisms.
\end{proof}

\begin{lemma} \label{lem:open-fact}
    Given a family of smooth curves $X/S$, and a fiberwise nonempty open subfamily $f : U/S \to X/S$, the corresponding map \[f : \Ran(U/S) \to \Ran(X/S)\] is both a map of commutative algebras and commutative coalgebras in $\PreSt^{\corr}$. In particular, both functors \[f_{\Ran,*} : \IndCoh(\Ran(U/S)) \leftrightarrows \IndCoh(\Ran(X/S)) : f_{\Ran}^!\] are symmetric monoidal with respect to the chiral monoidal structure. Furthermore, it induces maps \[f_{\Ran,*}^{\fact} : \Fact(U/S) \leftrightarrows \Fact(X/S) : f_{\Ran}^{!,\fact}~.\]
\end{lemma}

\begin{proof}
    We claim that we have a commutative diagram
    \begin{equation} \label{eq:comm-diag-bc-Ran}
    \begin{tikzcd}
	{\Ran(U/S)^{\times_S I}} & {\Ran(U/S)^{\times_S I}_{\disj}} & {\Ran(U/S)} \\
	{\Ran(X/S)^{\times_S I}} & {\Ran(X/S)^{\times_S I}_{\disj}} & {\Ran(X/S)}
	\arrow["{f^I_{\Ran}}", from=1-1, to=2-1]
	\arrow["{\jmath_U}"', from=1-2, to=1-1]
	\arrow["{\sqcup_U}", from=1-2, to=1-3]
	\arrow["{f_{\Ran,I}}", from=1-2, to=2-2]
	\arrow["f_{\Ran}", from=1-3, to=2-3]
	\arrow["\jmath"', from=2-2, to=2-1]
	\arrow["\sqcup", from=2-2, to=2-3]
    \end{tikzcd}
    \end{equation} 
    for each finite set $I$, and both squares are Cartesian. In particular, since compositions in correspondence categories are given by pullbacks, both multiplication maps commute with $f_{\Ran}$, and since $\IndCoh$ is lax-monoidal, it preserves maps of commutative algebras. Indeed, the claim follows since two collections of sections are disjoint in $U$ if and only if their image in $X$ is disjoint. 
    
    To show that $f_{\Ran}^!$ sends factorization algebras to factorization algebras, note that the factorization isomorphism for a pullback is the pullback of the factorization isomorphism, as in the proof of Lemma~\ref{lem:base-change-fact}. For $f_{\Ran,*}$, the claim requires the fact that both squares are Cartesian. Therefore, for $\cA \in \Fact(U/S)$, we have \[\sqcup^!f_{\Ran,*}\cA \simeq f_{\Ran,I,*}\sqcup_U^!\cA \iso f_{\Ran,I,*}\jmath_U^!\cA^{\otimes_S^!I} \simeq \jmath^!(f_{\Ran,*}\cA)^{\otimes_S^!I}\] where the first isomorphism follows from base-change for the right square, and the last isomorphism from base-change for the left square. 
\end{proof}

\begin{definition}
    Given a family of smooth curves $X/S$, a family of chiral algebras over $X/S$ is a Lie algebra \[\cB \in \Lie(\IndCoh^{\ch}(\Ran (X/S)))\]
    such that its underlying sheaf is given by $\Delta_!\cB_1$ for some $\cB_1 \in \IndCoh(X_{S\mh\dR})$. We let $\Ch(X/S)$ denote the category of families of chiral algebras. 
\end{definition}

Any factorization algebra defines a chiral algebra in the following way: Given a factorization algebra $\cA$ over a smooth curve $X_0 / \Spec k$, the D-module $\cA_{1}[-1] = \Delta_1^!\cA[-1]$ has a structure of a chiral algebra, with the Lie algebra map given by the boundary (or gluing) morphism associated to the sequence
\begin{equation} \label{eq:exact-seq-fact}
    \Delta_*\Delta^!(\cA_{1} \boxtimes \cA_{1}) \to \cA_{1} \boxtimes \cA_{1} \to \jmath_*\jmath^!(\cA_{1} \boxtimes \cA_{1})~.
\end{equation}

Explicitly, since $\cA$ is a factorization algebra, we have $\jmath^!(\cA \boxtimes \cA) \simeq \jmath^!\cA_{2}$. By \eqref{eq:Ran-iso}, we have $\cA_{1} \iso \Delta^!\cA_{2}$. Therefore, \eqref{eq:exact-seq-fact} becomes
\[\Delta_*\cA_{1} \to \cA_{1} \boxtimes \cA_{1} \to \jmath_*\jmath^!(\cA \boxtimes \cA)\]
and the corresponding chiral multiplication map is the boundary morphism, applied to the shift of this sequence:
\[\jmath_*\jmath^!(\cA_1[-1] \boxtimes \cA_1[-1]) = \jmath_*\jmath^!(\cA_1 \boxtimes \cA_1)[-2] \to \Delta_*\cA_{1}[-1]~.\]

The equivalence between chiral algebras and factorization algebras over a curve in the absolute case was proven in \cite{BD}. This result was generalized to arbitrary smooth schemes in \cite{FG}, where it was shown to be a special case of Koszul duality between Lie algebras and commutative coalgebras. By adapting the proof of~\cite{FG} to the relative case, we get:
\begin{theorem}
    For a family of smooth curves $X/S$, the Chevalley-Eilenberg complex defines an isomorphism \[\Chev_S^{\ch} : \Lie(\IndCoh^{\ch}(\Ran (X/S))) \iso \coComm(\IndCoh^{\ch}(\Ran (X/S)))\]
    and restricts to an isomorphism \[\Ch(X/S) \iso \Fact(X/S)~.\]
\end{theorem} 

We will denote this correspondence by 
\begin{align*}
    \cB & \mapsto \cB^{\fact} \\
    \cA^{\ch} & \mapsfrom \cA~.
\end{align*}

\subsection{Chiral and factorization modules}

Let $S$ be an algebraic stack locally of finite type. The chiral monoidal structure over $S$ was defined as a commutative algebra structure on $\Ran (X/S) \in \PreSt^{\corr}_S$. We therefore have a natural notion of modules over $\Ran (X/S)$. Such a module is given by an $S$-prestack $\cY / S$, and a commutative action map \[\cY \times_S \Ran (X/S) \from \Act_{\cY} \to \cY\] for some $S$-prestack $\Act_{\cY}$. Let $\Ran(X/S)\mh\Mod$ be the category of $\Ran(X/S)$-modules in $\PreSt_S^{\corr}$. In the case of augmented modules, the prestack $\Act_{\cY}$ is uniquely determined:
\begin{definition}
    Let $\Ran(X/S)\mh\Mod^{\aug}$ be the category of augmented modules $\cY \to \Ran(X/S)$, namely commutative diagrams \[\begin{tikzcd}
	{\cY \times_S \Ran(X/S)^{\times_SI}} & {(\cY \times_S \Ran(X/S)^{\times_S I})_{\disj}} & \cY \\
	{\Ran(X/S) \times_S \Ran(X/S)^{\times_SI}} & {(\Ran(X/S) \times_S \Ran(X/S)^{\times_S I})_{\disj}} & {\Ran(X/S)}
	\arrow[from=1-1, to=2-1]
	\arrow["{\jmath_{\cY}}"', from=1-2, to=1-1]
	\arrow["{\sqcup_{\cY}}", from=1-2, to=1-3]
	\arrow[from=1-2, to=2-2]
	\arrow[from=1-3, to=2-3]
	\arrow["\jmath"', from=2-2, to=2-1]
	\arrow["\sqcup", from=2-2, to=2-3]
    \end{tikzcd}\] for each finite set $I$, which induce a map of commutative algebras in $\PreSt^{\corr}_S$. In particular, this implies \[(\cY \times_S \Ran(X/S)^I)_{\disj} = {(\Ran(X/S) \times_S \Ran(X/S)^I)_{\disj}} \times_{(\Ran(X/S) \times \Ran(X/S))} (\Ran(X/S) \times \cY)~.\]
\end{definition}

\begin{definition}
    Let $\Ran(X/S)\mh\Mod^{\et}$ be the full subcategory of modules $\cY$ with action \[\cY \times_S \Ran(X/S) \xleftarrow{\jmath_{\cY}} \Act_{\cY} \xrightarrow{\sqcup_{\cY}} \cY\] such that $\sqcup_{\cY}$ is \'{e}tale. In particular, $\sqcup_{\cY}^!$ is left adjoint to $\sqcup_{\cY,*}$. 
\end{definition}

\begin{example}
    Any $\cY \in \Ran(X/S)\mh\Mod^{\aug}$ is automatically in $\Ran(X/S)\mh\Mod^{\et}$, since in that case $\sqcup_{\cY}$ is given by base-change of the \'{e}tale map $\sqcup$ for $\Ran(X/S)$.
\end{example}

\begin{definition}
    Given a family of factorization algebras $\cA \in \Fact(X/S)$ and a module space $\cY \in \Ran(X/S)\mh\Mod^{\et}$, we get an $\IndCoh^{\ch}(\Ran(X/S))$-module structure on $\IndCoh(\cY)$. Define $\cA\mh\FactMod(\cY)$ to be the full subcategory of $\cA\mh\coMod(\IndCoh(\cY))$ spanned by comodules $\cM$ such that the adjoint maps \begin{equation} \label{eq:fact-iso-mod}
    \sqcup_{\cY}^!\cM \to \jmath_{\cY}^!(\cM \otimes_S^! \cA^{\otimes_S^! I})
    \end{equation} are isomorphisms for each $I$.
\end{definition}

\begin{example} \label{ex:vac-mod}
    For an augmented module $\phi : \cY \to \Ran(X/S)$ and a factorization algebra $\cA \in \Fact(X/S)$, the pullback\[\Vac_{\cY}(\cA) = \phi^!\cA \in \cA\mh\FactMod(\cY)\] has a structure of a factorization module in $\cY$. This is a special case of Lemma~\ref{lem:module-map-fact} below.
\end{example}

\begin{lemma} \label{lem:module-map-fact}
    Given a family of smooth curves $X/S$, modules $\cY,\cZ \in \Ran(X/S)\mh\Mod^{\et}$, and a map of prestacks $f : \cY \to \cZ$, which is a morphism of $\Ran(X/S)$-modules, we have a restriction map \[f^{\cA,!} : \cA\mh\FactMod(\cZ) \to \cA\mh\FactMod(\cY)\] Assuming further that $f$ is also a $\Ran(X/S)$-comodule map, with $\Ran(X/S)$ viewed as a commutative coalgebra, then we also have an extension map \[f^{\cA}_* : \cA\mh\FactMod(\cY) \to \cA\mh\FactMod(\cZ)~.\] 
\end{lemma}

\begin{proof}
    If $f$ is a module map, we have commutative diagrams  \[\begin{tikzcd}
	{\cY \times_S \Ran(X/S)^{\times_SI}} & {\Act_{\cY}} & \cY \\
	{\cZ \times_S \Ran(X/S)^{\times_SI}} & {\Act_{\cZ}} & \cZ
	\arrow["{f \times \id}", from=1-1, to=2-1]
	\arrow["{\jmath_{\cY}}"', from=1-2, to=1-1]
	\arrow["{\sqcup_{\cY}}", from=1-2, to=1-3]
	\arrow["{f_{I}}", from=1-2, to=2-2]
	\arrow["f", from=1-3, to=2-3]
	\arrow["{\jmath_{\cZ}}"', from=2-2, to=2-1]
	\arrow["{\sqcup_{\cZ}}", from=2-2, to=2-3]
    \end{tikzcd}\] such that the left square is Cartesian. Given $\cM \in \cA\mh\FactMod(\cZ)$ we have commutative isomorphisms \[\sqcup_{\cZ}^!\cM \iso \jmath_{\cZ}^!(\cM \otimes_S^! \cA^{\otimes_S^! I})\] and by pullback, commutative isomorphisms \[\sqcup_{\cY}^!f^!\cM \simeq f_I^!\sqcup_{\cZ}^!\cM \iso f_I^!\jmath_{\cZ}^!(\cM \otimes_S^! \cA^{\otimes_S^! I}) \simeq \jmath_{\cY}^!(f^!\cM \otimes_S^! \cA^{\otimes_S^! I})~.\] The coaction maps are given by the adjoint maps \[f^!\cM \to \sqcup_{\cY,*}\jmath_{\cY}^!(f^!\cM \otimes_S^! \cA^{\otimes_S^! I})~.\]

    Assuming $f$ is also a comodule map, the right square is Cartesian as well. Given $\cM \in \cA\mh\FactMod(\cY)$, we have commutative isomorphisms \[\sqcup_{\cY}^!\cM \iso \jmath_{\cY}^!(\cM \otimes_S^! \cA^{\otimes_S^! I})~.\] Pushing forward and using base-change for both squares, we get commutative isomorphisms \[\sqcup_{\cZ}^!f_*\cM \simeq f_{I,*}\sqcup_{\cY}^!\cM \iso f_{I,*}\jmath_{\cY}^!(\cM \otimes_S^! \cA^{\otimes_S^! I}) \simeq \jmath_{\cZ}^!(f_*\cM \otimes_S^! \cA^{\otimes_S^! I})\] The coaction is given by the adjoint maps \[f_*\cM \to \sqcup_{\cZ,*}\jmath_{\cZ}^!(f_*\cM \otimes_S^! \cA^{\otimes_S^! I})~.\]
\end{proof}

An important class of modules is that of modules supported at a point, or more generally on any $Z \to \Ran(X/S)$. 
\begin{definition}
    Let $\Ran^{\subset}(X/S)$ be the $S$-prestack whose $\Spec R \to S$ points are given by finite nonempty sets of lifts $A \subset B \subset X_{S\mh\dR}(R)$. It can be written as a colimit \[\Ran^{\subset}(X/S) \simeq \colim_{J \subset I \in \fSet^{\subset,\surj,\op}} X^{\times_SI}_{S\mh\dR}\]
    where $\fSet^{\subset,\surj}$ is the category of pairs of finite sets $J \subset I$ and surjective morphisms of pairs between them.

    We have two projections \[\pr_s, \pr_b : \Ran^{\subset}(X/S) \rightrightarrows \Ran (X/S)\] given by $\pr_s(A \subset B) = A$, and $\pr_b(A \subset B) = B$.
\end{definition}

\begin{example} \label{ex:fact-mod}
    Given any $Z \to \Ran (X/S)$, let \[\Ran_Z(X/S) = Z \times_{\Ran (X/S), \pr_s} \Ran^{\subset} (X/S) \]
    Then \[\Ran_Z(X/S) \times \Ran(X/S)^{\times_SI} \xleftarrow{\jmath_Z} (\Ran_Z(X/S) \times_S \Ran(X/S)^{\times_SI})_{\disj} \xrightarrow{\sqcup_Z} \Ran_Z(X/S)\] defines a structure of an augmented $\Ran (X/S)$-module, with the augmentation map given by the projection. Here 
    \begin{align*}
        & (\Ran_Z(X/S) \times_S \Ran(X/S)^{\times_SI})_{\disj} \\
        \coloneqq & (\Ran_Z(X/S) \times_S \Ran (X/S)^{\times_S I}) \times_{\Ran (X/S) \times_S \Ran (X/S)^{\times_S I}} (\Ran (X/S) \times_S \Ran (X/S)^{\times_S I})_{\disj}
    \end{align*} $\jmath_Z$ is the projection, the map $\Ran_Z(X/S) \to \Ran (X/S)$ is given by \[\Ran_Z(X/S) \to \Ran^{\subset}(X/S) \xrightarrow{\pr_b} \Ran (X/S)\] and $\sqcup_Z$ is the map $(A \subset B, B') \mapsto (A \subset B \sqcup B')$. 
    
    As above, we get an $\IndCoh^{\ch}(\Ran(X/S))$-module structure on $\IndCoh(\Ran_Z(X/S))$, and, for $\cA \in \Fact(X/S)$, a category of $\cA$-modules in $\IndCoh(\Ran_Z(X/S))$. Denote \[\cA\mh\FactMod_Z \coloneqq \cA\mh\FactMod(\IndCoh(\Ran_Z(X/S)))\] This is the category of $\cA$-modules supported at $Z$.
\end{example}

Now recall from Lemma~\ref{lem:base-change-fact} that for $f: T \to S$, we have a restriction map \[u^{!,\fact}_{f} : \Fact(X/S) \to \Fact(X_T/T)~.\]

\begin{lemma} \label{lem:base-change-fact-mod}
    For $f : T \to S$ and a family of smooth curves $X/S$ compatible with base-change along $f$ as in Definition~\ref{def:Ran-compat-bc}, we have an induced map \[(-)_T : \Ran(X/S)\mh\Mod^{\et} \to \Ran(X_T/T)\mh\Mod^{\et}\] For $\cY \in \Ran(X/S)\mh\Mod^{\et}$, and $\cA \in \Fact(X/S)$, we have an induced map \[u_{f}^{!,\cA} : \cA\mh\FactMod(\cY) \to \cA\mh\FactMod(\cY_T)~.\] For $\cY \in \Ran(X_T/T)\mh\Mod^{\et}$ and $\cA \in \Fact(X_T/T)$, we have an induced map \[u_{f,*}^{\cA} : \cA\mh\FactMod(\cY_T) \to u_{f,*}^{\fact}\cA\mh\FactMod(\cY)~.\]
\end{lemma}

\begin{proof}
    As mentioned in the proof of Lemma~\ref{lem:base-change-fact}, $(-)_T$ is symmetric monoidal, and thus induces a map on commutative algebras and modules over them. Furthermore, given $\cY \in \Ran(X/S)\mh\Mod^{\et}$, the action map $\sqcup_T : \Act_{\cY,T} \to \cY_T$ is given by pullback of the \'{e}tale map $\sqcup : \Act_{\cY} \to \cY$, and therefore is also \'{e}tale. Since both $u_{f,*}$ and $u_f^!$ are symmetric monoidal, we get by functoriality an induced map on modules, so we only need to show that the factorization isomorphism~\eqref{eq:fact-iso-mod} holds in both cases. The proof is similar to that of Lemma~\ref{lem:base-change-fact}.
\end{proof}

By Lemma~\ref{lem:open-fact}, we also have extensions and restrictions to open sets, namely, for an open subfamily $f : U/S \to X/S$, we have maps \[f_{\Ran,*}^{\fact} : \Fact(U/S) \leftrightarrows \Fact(X/S) : f_{\Ran}^{!,\fact}~.\]

\begin{lemma} \label{lem:open-fact-mod}
    Given an open subfamily $f : U/S \to X/S$, the corresponding map $f_{\Ran} : \Ran(U/S) \to \Ran(X/S)$ induces a map \[(-)_U : \Ran(X/S)\mh\Mod^{\et} \to \Ran(U/S)\mh\Mod^{\et}~,\] which commutes with the forgetful functor to $\PreSt_S^{\corr}$. Given $\cY \in \Ran(X/S)\mh\Mod^{\et}$ and $\cA \in \Fact(X/S)$, we have an induced map \[f_{\Ran}^{!,\cA} : \cA\mh\FactMod(\cY) \to f_{\Ran}^{!,\fact}\cA\mh\FactMod(\cY)~.\]
\end{lemma}

\begin{proof}
    Since $f_{\Ran}$ is a map of commutative algebras, the map \[\Ran(X/S)\mh\Mod \to \Ran(U/S)\mh\Mod\] is given by restriction. Namely, the action is defined via the commutative diagrams
    \[\begin{tikzcd}
	{\cY \times_S \Ran(U/S)^{\times_SI}} & {\Act_{\cY} \underset{\cY \times\Ran(X/S)^{\times_SI}}{\times} \cY \times_S \Ran(U/S)^{\times_SI}} & {\cY} \\
	{\cY \times_S \Ran(X/S)^I} & {\Act_{\cY}} & {\cY}
	\arrow["{\id \times f^I_{\Ran}}", from=1-1, to=2-1]
	\arrow["{\jmath_{\cY,U}}"', from=1-2, to=1-1]
	\arrow["{\sqcup_{\cY,U}}", from=1-2, to=1-3]
	\arrow["{f_{\Ran,I}}", from=1-2, to=2-2]
	\arrow["\id", from=1-3, to=2-3]
	\arrow["\jmath_{\cY}"', from=2-2, to=2-1]
	\arrow["\sqcup_{\cY}", from=2-2, to=2-3]
    \end{tikzcd}\] Furthermore, since $\sqcup_{\cY}$ is \'{e}tale and $f_{\Ran,I}$ is open, $\sqcup_{\cY,U}$ is \'{e}tale as well. By functoriality, we get an induced map on modules, and the factorization isomorphism for the pullback $f_{\Ran}^!\cM$ is given by pulling the factorization isomorphism for $\cM$, as in the proof of Lemma~\ref{lem:open-fact}.
\end{proof}

If $\cY = \Ran_Z(X/S)$, as in Example~\ref{ex:fact-mod}, we also have a notion of chiral modules:
\begin{definition}
    Given a family of chiral algebras $\cB \in \Ch(X/S)$, and a map $Z \to \Ran(X/S)$, a Lie module $\cN \in \Mod_{\cB}(\IndCoh(\Ran_Z(X/S)))$ is said to be a chiral $\cB$-module supported at $Z$ if its underlying ind-coherent sheaf is supported on $Z \subset \Ran_ZX$. Denote the resulting category by $\cB\mh\ChMod_{Z}$.
\end{definition}

By adapting the proof of \cite{BD} and \cite{FG}  to the relative case, we get:
\begin{theorem} \label{thm:CKD}
    The Chevalley-Eilenberg complex defines an isomorphism \[\Chev^{\ch} : \cB\mh\ChMod_{Z} \iso \cB^{\fact}\mh\FactMod_{Z}~.\]
\end{theorem}

Given a chiral module $\cN$, we will denote the corresponding factorization module by \[\cN^{\fact} = \Chev^{\ch}(\cN)~.\]

\subsection{Chiral Homology}

\begin{definition}
    Assume now $p : X \to S$ is proper, and denote the induced map on the Ran space by \[p_{\Ran} : \Ran(X/S) \to S\] The chiral (or factorization) homology of $X/S$ with coefficients in $\cA$ is defined as its global sections \[\int_{X/S} \cA = p_{\Ran,!}\cA \in \IndCoh(S)~.\]

    More generally, given  $Z \to \Ran (X/S)$ and $\cM \in \cA\mh\ChMod_{Z}$, let $p_{Z,\Ran} : \Ran_Z(X/S) \to S$. Then we define the chiral homology over $(X/S,Z)$ with coefficients in $(\cA,\cM)$ to be \[\int_{(X/S,Z)} (\cA,\cM) := p_{Z,\Ran,!}\cM^{\fact} \in \IndCoh(S)~.\]
\end{definition}

The existence of $p_{\Ran,!}$ follows from the fact that in the case where $X \to S$ is proper, $\Ran (X/S) \to S$ is pseudo-proper, and in particular $p_{\Ran,*} = p_{\Ran,!}$ is left adjoint to $p_{\Ran}^!$. From \eqref{eq:colim-present}, we get \[\int_{X/S} \cA \simeq \colim_{I \in \fSet^{\surj,\op}} p_{I,!}\cA_{I}\] where $p_I : X^{\times_SI}_{S\mh\dR} \to S$ is the projection. Similarly, we have \[\int_{(X/S,Z)} (\cA,\cM) \simeq \colim_{J \subset I \in \fSet^{\subset,\surj,\op}} p_{Z,I,!}\cM^{\fact}_{I}~.\]

\subsection{Universal factorization algebras}

Let $X$ be a curve. The local structure of a chiral algebra $\cB$ with underlying D-module $\cB_1 \in D\mh\Mod(X)^{\heartsuit}$ around a smooth point $x \in X$ can be described in terms of vertex operator algebras. A vertex operator algebra (VOA) is a graded vector space $V$ equipped with a derivation $T$, two distinguished vectors $\pmb{1}, \omega \in V$, and a $D^{\times}$-family of multiplications \[Y(-,z) : V \to \operatorname{End}(V)[[z^{\pm 1}]]\] satisfying a list of axioms. The vector $\pmb{1}$ is the vacuum vector, which behaves like an identity element, while $\omega$ is the conformal vector, which defines an action of the Virasoro algebra, and therefore strong equivariance with respect to formal deformations. Given a VOA, one can construct a global object {--} a chiral algebra over the smooth locus $X^{\sm}$, which has $V$ as its local structure around any smooth point. This was done in \cite{FBZ}. Such a chiral algebra can be defined over any curve in a compatible way. The resulting structure is that of a \textit{universal chiral algebra} {--} a chiral algebra that is defined over any smooth curve and is compatible with \'{e}tale pullbacks (Cf.~\cite[Section~3.1.16]{BD} and \cite{Cliff}). The construction is done using Gelfand-Kazhdan descent, namely using pullback along the map \[X \to B\Aut_*\hat{\cO}\] obtained from the principal $\Aut_*\hat{\cO}$-bundle $\operatorname{Coord}_X \to X$ whose fiber at a point is given by the set of formal coordinates around the point. We will define here a similar construction that gives a derived version of vertex operator algebras, by defining a "chiral product" on the category of $\Aut\hat{\cO}$-representations. This approach is borrowed from~\cite{Lur}:

\begin{definition} \label{def:multi_disks}
    For a $k$-algebra $R$, a \textit{formal multidisk} over $R$ is a formal $R$-scheme which, \'{e}tale locally, is given by the formal completion of $\AA_R^1$ along a finite and flat divisor. For a finite set $I$, let $\MDisk_I^{\str}(R)$ be the space of $R$-multidisks together with a \textit{strict $I$-pointing}, namely an $R$-multidisk $\cX$ together with a set-theoretically surjective map \[\sqcup_{i \in I} x_i : (\Spec R)^{\sqcup I} \to \cX\] Let $\MDisk_I(R)$ be the space of $R$-multidisks equipped with an $I$-pointing, namely a surjective map \[\sqcup_{i \in I} x_i : (\Spec R^{\red})^{\sqcup I} \to \cX~.\]
\end{definition}

\begin{example}
    For $R = k$, any $I$-pointed $k$-multidisk is given by a disjoint union $\Spf k[[t]]^{\sqcup J}$ for some $|J| \leq |I|$, up to an isomorphism.
\end{example}

\begin{example} \label{ex:MDisk_pt}
    For $I = \{\pt\}$, $\MDisk_I$ is given by $B\Aut\hat{\cO}$, and $\MDisk_I^{\str}$ is its subgroup of basepoint-preserving automorphisms $B\Aut_*\hat{\cO}$.
\end{example}

In general, however, the number of disks in each geometric fiber of an $R$-multidisk may not be constant. 

\begin{example} \label{ex:X-to-disk}
    For any flat family of smooth curves $X/S$ over an algebraic stack $S$ and a map $\Spec R \to X^{\times_SI}$, let $X_R = X \times_{S} \Spec R$. Then the fiberwise formal completion $X_R \times_{(X_R)_{R\mh\dR}} \Spec R^{\sqcup I}$ defines a strictly $I$-pointed $R$-multidisk. This defines a map \[\pi_I^{X/S} : X^{\times_SI} \to \MDisk^{\str}_I\] Similarly, given a map $\Spec R \to X^{\times_SI}_{S\mh\dR}$, the formal completion defines an $I$-pointed $R$-multidisk. This defines a map \[\pi_{I,\dR}^{X/S} : X^{\times_SI}_{S\mh\dR} \to \MDisk_I~.\]
\end{example}

\begin{definition}
    For a surjective map $\alpha : J \onto I$, let \[\Delta_{\alpha} : \MDisk_I^{\str} \to \MDisk_J^{\str}\] be the map sending an $I$-pointed $R$-multidisk $(\cX, \{x_i\}_{i \in I})$ to $(\cX, \{x_{\alpha(j)}\}_{j \in J})$, and \[\sqcup_{(\alpha)} : \prod_{i \in I} \MDisk_{J_i}^{\str} \to \MDisk_{J}^{\str}\] the disjoint union map. Define $\Delta_{(\alpha)} : \MDisk_I \to \MDisk_J$ and $\sqcup_{(\alpha)} : \prod_{i \in I} \MDisk_{J_i} \to \MDisk_{J}$ similarly.
\end{definition}

The assignments \[I \mapsto \MDisk_I^{\str},\MDisk_I, \quad \alpha \mapsto \Delta_{(\alpha)}\] define lax-monoidal functors \[\fSet^{\surj,\op} \to \operatorname{Stacks} \into \PreSt\] with respect to the disjoint union on the source and the Cartesian product on the target, with the lax structure given by \[\sqcup : \prod_{j \in J}\MDisk_{I_j} \to \MDisk_{\sqcup_j I_j}~.\]

\begin{definition}
    Let \[\MDisk_I^{\bullet} \to \MDisk_I\] be the \v{C}ech nerve of the map $\MDisk_I^{\str} \to \MDisk_I$.
\end{definition}

\begin{lemma} \label{lem:ff}
    The maps \[\tilde{\pi}^{\AA^1}_{\dR,I} : \AA^I \to \MDisk_I\] defined by Example~\ref{ex:X-to-disk} are faithfully flat, schematic, qcqs, and relatively apaisant (see Definition~\ref{def:apai}).
\end{lemma}

\begin{proof}
    The map is a surjection by the definition of multidisks, so we have to prove flatness. Given a map $\Spec R \to \MDisk_I$ parametrizing an $I$-pointed $R$-multidisk $(\cX,\{x_i\}_{i \in I})$, the fiber is given by \[\AA^{I}_{\dR} \times_{\MDisk_I} \Spec R \simeq \operatorname{Emb}_R(\cX,\AA^1_R)\] the space of $R$-linear embeddings. The latter is an open subspace of the mapping stack $\underline{\operatorname{Map}}_R(\cX,\AA^1_R)$. By potentially restricting to an \'{e}tale cover $\Spec R' \to \Spec R$, we may assume $\cX$ is given by the formal completion $\hat{Z} \subset \AA_R^1$ along a closed subscheme $Z \subset \AA_R^1$, flat over $R$, defined by a monic polynomial $f(x) = x^m + \alpha_{m-1}x^{m-1} + \dotsb + \alpha_0 \in R[x]$. Then \[\cX \simeq \hat{Z} \simeq \colim_n Z^{(\leq n)} = \colim_n \Spec R[x]/(f)^n\] Therefore, the above mapping stack can be written as a limit \[\underline{\operatorname{Map}}_R(\cX,\AA^1_R) \simeq \lim_n \underline{\operatorname{Map}}_R(Z^{(\leq n)},\AA^1_R) \simeq \lim_n \underline{\operatorname{Map}}_R(R[x],R[x]/(f)^n) \simeq \lim_n \AA_R^{nm}~,\] where the last isomorphism identifies a morphism with the coefficients of its image of $x$: \[p(x) = \sum \beta_ix^i \mapsto (\beta_0,\beta_1,\dots)~.\] Such a map is an embedding iff it is a bijection on closed points, and induces an isomorphism on formal neighborhoods. The first condition depends only on the first $m$ coefficients, and is therefore given by an open subscheme $U'$ of $\AA_R^m$, and the latter is given by imposing the condition that $\beta_1$ is invertible. Let $U = U' \cap \AA^1_R \times \GG_{m,R} \times \AA^{m-2}_R$. We get \[\underline{\operatorname{Emb}}_R(\cX,\AA^1_R) \simeq U \times \lim_n \AA^{(m-1)n}_R\] which is qcqs and apaisant. It is left to show that each $\underline{\operatorname{Map}}_R(\cX,\AA^1_R)$ is flat over $R$. Since flatness is preserved under limits, it suffices to show that each $\underline{\operatorname{Map}}_R(Z^{(\leq n)},\AA^1_R)$ is flat over $R$. Since these are affine schemes of finite type, it suffices to show they are formally smooth. Given any ring $B$, a square-zero ideal $J \subset B$, and a commutative diagram \[\begin{tikzcd}
	{\Spec B/J} & {\underline{\operatorname{Map}}_R(Z^{(\leq n)},\AA^1_R)} \\
	{\Spec B} & {\Spec R}
	\arrow["f_0", from=1-1, to=1-2]
	\arrow[from=1-1, to=2-1]
	\arrow[from=1-2, to=2-2]
	\arrow["\tilde{f}", dashed, from=2-1, to=1-2]
	\arrow["f", from=2-1, to=2-2]
    \end{tikzcd}\] we need to show the existence of a lift $\tilde{f}$, or equivalently a lift \[\begin{tikzcd}
	{Z^{(\leq n)} \times_{\Spec R} \Spec B/J} & {\AA^1_R} \\
	{Z^{(\leq n)} \times_{\Spec R} \Spec B} & {\Spec R}
	\arrow[from=1-1, to=1-2]
	\arrow[from=1-1, to=2-1]
	\arrow[from=1-2, to=2-2]
	\arrow[dashed, from=2-1, to=1-2]
	\arrow[from=2-1, to=2-2]
    \end{tikzcd}\] which exists since $\AA^1_R \to \Spec R$ is smooth.
\end{proof}

\begin{lemma}
    The maps $\MDisk_I^{\str} \to \MDisk_I$ are faithfully flat.
\end{lemma}

\begin{proof}
    We have a Cartesian diagram \begin{equation} \label{eq:MDisk-cover}
        \begin{tikzcd}
	{\AA^I} & {\MDisk_I^{\str}} \\
	{\AA^I_{\dR}} & {\MDisk_I}
	\arrow[from=1-1, to=1-2]
	\arrow[from=1-1, to=2-1]
	\arrow[from=1-2, to=2-2]
	\arrow[from=2-1, to=2-2]
    \end{tikzcd}
    \end{equation} where the bottom arrow is faithfully-flat by Lemma~\ref{lem:ff}, and the left arrow is faithfully flat since $\AA^I$ is smooth. In particular, the right arrow is faithfully flat as well.
\end{proof}

\begin{corollary}
    The stacks $\MDisk_I^{\bullet}$ are weakly renormalizable in the sense of Definition~\ref{def:weakly-ren}. In particular, we have compactly generated categories $\IndCoh^!_{\ren}(\MDisk_I^{\bullet})$. 
\end{corollary}

\begin{proof}
    From diagram~\eqref{eq:MDisk-cover}, we get a faithfully flat map of complexes \[\AA^{I,\bullet}_{\dR} \to \MDisk^{\bullet}_I\] where $\AA^{I,\bullet}_{\dR}$ is the infinitesimal groupoid of $\AA^I$ (namely the \v{C}ech nerve of $\AA^I \to \AA^I_{\dR}$), and each term in the LHS is an ind-scheme of ind-finite type. 
\end{proof}

\begin{remark}
    Taking $I = \{\pt\}$ and $[n] = [0]$, we get by Example~\ref{ex:MDisk_pt} the category $\IndCoh_{\ren}^!(B\Aut_*\hat{\cO})$. As in~\cite[Example~8.21.2]{RasHM}, this category can be identified with the twist of $\IndCoh^{\ren}_*(B\Aut\hat{\cO}) = \Rep(\Aut\hat{\cO})$ by the Tate central extension.
\end{remark}

\begin{definition} \label{def:multi_disk_Ran}
    Let \[\MDisk_{\Ran}^{\bullet} = \underset{\fSet^{\surj,\op}}{\colim} \MDisk_{I}^{\bullet}\] and \[\MDisk_{\Ran} = \underset{\fSet^{\surj,\op}}{\colim} \MDisk_{I}\] Define \[\IndCoh_{\ren}^!(\MDisk_{\Ran}^{\bullet}) \coloneqq \lim_{\fSet^{\surj}} \IndCoh_{\ren}^!(\MDisk_I^{\bullet})\] and \[\IndCoh_{\ren}^!(\MDisk_{\Ran}) \coloneqq \operatorname{Tot}(\IndCoh^!_{\ren}(\MDisk_{\Ran}^{\bullet}))~,\] where the limits are taken with respect to $!$-pullbacks. For $\alpha : I \onto J$, let \[\Delta_{(\alpha),!} : \IndCoh_{\ren}^!(\MDisk_J) \to \IndCoh_{\ren}^!(\MDisk_I)\] and \[\sqcup_{(\alpha),?} : \IndCoh_{\ren}^!(\prod_{j \in J} \MDisk_{I_j}) \to \IndCoh_{\ren}^!(\MDisk_J)\] be the limits of the corresponding functors for $\MDisk_I^{\bullet}$, which are defined by Lemma~\ref{lem:ren-!push} and Lemma~\ref{lem:ren-?push}. 
\end{definition}

\begin{lemma}
    $\IndCoh_{\ren}^!(\MDisk_{\Ran})$ admits a symmetric monoidal structure $\otimes^{\ch}$ induced by the disjoint union maps. We denote the resulting symmetric monoidal category by $\IndCoh_{\ren}^{!,\ch}(\MDisk_{\Ran})$.
\end{lemma}

\begin{proof}
    First note that we can write \[\IndCoh_{\ren}^!(\MDisk^{\bullet}_{\Ran}) = \lim_{\fSet^{\surj}} \IndCoh_{\ren}^!(\MDisk_I^{\bullet})\simeq \underset{\fSet^{\surj,\op}}{\colim} \IndCoh_{\ren}^!(\MDisk_I^{\bullet})~,\] where the colimit is taken with respect to $!$-pushforward. Indeed, since $\Delta_{(\alpha)}$ are closed embeddings, $\Delta_{(\alpha)}^!$ admit left adjoints $\Delta_{(\alpha),!}$ by Lemma~\ref{lem:ren-!push}. The resulting functor \[\fSet^{\surj,\op} \to \Cat_k\] given by \[I \mapsto \IndCoh^{!}_{\ren}(\MDisk_I^{\bullet}), \; \alpha \mapsto \Delta_{(\alpha),!}\] is then lax-monoidal, with the lax structure given by the maps \[\otimes_{j \in J}\IndCoh^!_{\ren}(\MDisk_{I_j}^{\bullet}) \xrightarrow{\boxtimes} \IndCoh^!_{\ren}(\prod_{j \in J}\MDisk^{\bullet}_{I_j}) \xrightarrow{\sqcup_{?}} \IndCoh^!_{\ren}(\MDisk^{\bullet}_I)~.\] Here $\sqcup_?$ is the continuous right adjoint to $\sqcup^!$, which is defined by Lemma~\ref{lem:ren-?push}. Thus, each term in the colimit $\IndCoh^!_{\ren}(\MDisk^{\bullet}_{\Ran})$ admits a canonical symmetric monoidal structure, and the morphisms in this diagram commute with the symmetric monoidal structure. Finally, the symmetric monoidal structure on $\IndCoh_{\ren}^!(\MDisk_{\Ran})$ is defined as the limit of $\IndCoh^!_{\ren}(\MDisk^{\bullet}_{\Ran})$ in $\Comm(\Cat_k)$, which commutes with the forgetful functor to $\Cat_k$.
\end{proof}

Explicitly, we can represent each $\cF \in \IndCoh^{!,\ren}(\MDisk_{\Ran})$ by a collection $\cF_{J} \in \IndCoh^{!,\ren}(\MDisk_J)$ for each finite set $J$, and identifications $\Delta_{(\alpha)}^!\cF_{J} \simeq \cF_{J'}$ for each $\alpha : J' \onto J$. Given $\cF_i \in \IndCoh^{!,\ren}(\MDisk_{\Ran})$, we have \[\otimes^{\ch}_{i}\cF_{i,J} = \bigoplus_{\beta : J \onto I} \sqcup_{(\beta),?} \boxtimes_{i \in I} \cF_{i,J_i}~.\]

\begin{remark}
    Note that, unlike $\MDisk_I^{\str}$, the stacks $\MDisk_{I}$ are not renormalizable, since the composition $\AA^I \to \MDisk_I$ is only \textit{ind}-schematic. Similarly the prestack $\MDisk_{\Ran}^{\bullet}$ (and certainly $\MDisk_{\Ran}$) are not renormalizable. Thus, the notations we used for their categories of sheaves $\IndCoh_{\ren}^!(\MDisk_{I})$ etc. are for convenience only, and are not meant to be understood as a case of renormalization of categories.
\end{remark}

\begin{definition}
    Let $\Ch^{\univ}$ be the full subcategory of $\Lie(\IndCoh^{!,\ch}_{\ren}(\MDisk_{\Ran}))$ spanned by Lie algebras $\cB$ whose underlying sheaf is given by $\cB_I = \Delta_{I,!}\cB_1$ for some \[\cB_1 \in \IndCoh^!_{\ren}(\MDisk_1) \simeq \IndCoh_{\ren}^!(B\Aut\hat{\cO})~.\] 
    
    Let $\Fact^{\univ}$ be the full subcategory of $\coComm(\IndCoh^{!,\ch}_{\ren}(\MDisk_{\Ran}))$ spanned by coalgebras $\cA$ such that the maps \begin{equation} \label{eq:univ-fact-iso}
        \sqcup^!\cA \to \cA^{\boxtimes I}
    \end{equation} adjoint to the comultiplication maps are isomorphisms.

    We refer to an object of $\Ch^{\univ}$ as a universal chiral algebra, and to an object of $\Fact^{\univ}$ as a universal factorization algebra.
\end{definition}

\begin{lemma} \label{lem:pro-nil}
    The chiral monoidal structure $\IndCoh^{!,\ch}_{\ren}(\MDisk_{\Ran})$ is pro-nilpotent in the sense of~\cite{FG}.
\end{lemma}

\begin{proof}
    By definition and commutation of limits, \begin{align*}
        \IndCoh^{!}(\MDisk_{\Ran}) & = \lim_{[n] \in \Delta} \lim_{I \in \fSet^{\surj}} \IndCoh(\MDisk^{[n]}_{I}) \\
        & \simeq \lim_{m \geq 0} \lim_{[n] \in \Delta} \lim_{|I| \leq m} \IndCoh(\MDisk^{[n]}_{I}) \\
        & \eqqcolon \lim_{m \geq 0} \IndCoh(\MDisk_{\leq m})~.
    \end{align*} Each of the categories $\IndCoh(\MDisk_{\leq m})$ inherits a symmetric monoidal structure, similar to the argument in~\cite[Section~5.1.5]{FG}. This monoidal structure is nilpotent --- the disjoint union of $> m$ multi-disks is necessarily zero in $\IndCoh(\MDisk_{\leq m})$. 
\end{proof}

From~\cite[Proposition~4.3.3]{FG} we get:
\begin{corollary}
    The Chevalley-Eilenberg functor defines an isomorphism \[\Lie(\IndCoh^{!,\ch}_{\ren}(\MDisk_{\Ran})) \iso \coComm(\IndCoh^{!,\ch}_{\ren}(\MDisk_{\Ran}))~.\]
\end{corollary}

We do not know whether this functor restricts to an isomorphism \[\Ch^{\univ} \iso \Fact^{\univ}\] However, for a universal chiral algebra $\cB$, the pullback of $\Chev^{\ch}(\cB)$ to any smooth curve will be a factorization algebra, as we describe below.

\begin{definition} \label{def:pi_Ran^XS}
    For a family of curves $X/S$, define \[\pi^{X/S}_{\dR} : \Ran(X/S) \to \MDisk_{\Ran}\] to be the colimit of the maps $\pi_{\dR,I}^{X/S}$.
\end{definition}

\begin{theorem} \label{thm:loc-fact}
    For a family of smooth curves $X/S$ over an algebraic stack $S$ locally of finite type, the $!$-pullback along $\pi_{\dR}^{X/S}$ is a symmetric monoidal functor \[\IndCoh^{!,\ch}_{\ren}(\MDisk_{\Ran}) \to \IndCoh(\MDisk_{\Ran} \times S) \to \IndCoh^{\ch}(\Ran(X/S))\] and induces maps \[\Ch^{\univ} \to \Ch(X/S),\quad \Fact^{\univ} \to \Fact(X/S)~.\] Furthermore, it commutes with the Chevalley-Eilenberg functor, so that we have a commutative diagram
    \[\begin{tikzcd}
	{\Ch^{\univ}} & {\coComm(\IndCoh^{!,\ch}_{\ren}(\MDisk_{\Ran}))} \\
	{\Ch(X/S)} & {\Fact(X/S)}
	\arrow["{\Chev^{\ch}}", from=1-1, to=1-2]
	\arrow["{\pi^{X/S,!}_{\dR}}", from=1-1, to=2-1]
	\arrow["{\pi^{X/S,!}_{\dR}}", from=1-2, to=2-2]
	\arrow["{\Chev^{\ch}_S}", from=2-1, to=2-2]
    \end{tikzcd}\]
\end{theorem}

\begin{proof}
    Given a surjective map $\alpha : I \onto J$, we have a commutative diagram \[\begin{tikzcd}
	{\prod_{j \in J}X^{\times_SI_j}_{S\mh\dR}} & {(\prod_{j \in J}X_{S\mh\dR}^{\times_SI_j})_{\disj}} & {X^{\times_SI}_{S\mh\dR}} \\
	{\prod_{j \in J}\MDisk_{I_j}} & {\prod_{j \in J}\MDisk_{I_j}} & {\MDisk_I}
	\arrow["{\prod_{j \in J} \pi_{\dR,I_j}}"', from=1-1, to=2-1]
	\arrow["\jmath"', from=1-2, to=1-1]
	\arrow["\sqcup", from=1-2, to=1-3]
	\arrow["\pi_{\dR,I}^{\disj}", from=1-2, to=2-2]
	\arrow["{\pi_{\dR,I}}"', from=1-3, to=2-3]
	\arrow["\id"', from=2-2, to=2-1]
	\arrow["\sqcup", from=2-2, to=2-3]
    \end{tikzcd}\] and the right square is Cartesian. Using base-change for the right square and Lemma~\ref{lem:ren-?push}, we get \begin{align*}
        (\pi_{\dR}^{X/S,!}\otimes_{j \in J}^{\ch} \cF_{j})_I & = \pi_{\dR,I}^{X/S,!}\bigoplus_{\alpha : I \onto J} \sqcup_{?}\boxtimes_{j \in J} (\cF_j)_{I_j} \\
        & \simeq \bigoplus_{\alpha : I \onto J} \sqcup_*\jmath^!\boxtimes_{j \in J} \pi^{X/S,!}_{\dR,I_j}(\cF_j)_{I_j} = (\otimes^{\ch}_{j \in J} \pi_{\dR}^{X/S,!}\cF_j)_I~.
    \end{align*} 
Here we are using the notation $\sqcup_*$ for $X/S$ rather than $\sqcup_?$ since the stacks involved are of finite type, and therefore $\sqcup_*$ is defined and equals the right adjoint of $\sqcup^! = \sqcup^*$. By Lemma~\ref{lem:ren-!push}, $\pi_{\dR,I}^{X/S,!}\Delta_{I,!} \simeq \Delta_{I,!}\pi_{\dR,1}^{X/S,!}$, and so a chiral algebra is sent to a chiral algebra. For a universal factorization algebra $\cA$, the factorization isomorphism~\eqref{eq:fact-iso} for $\pi^{X/S,!}_{\dR}\cA$ follows from the factorization isomorphism~\eqref{eq:univ-fact-iso}. Finally, since $\pi^{X/S,!}_{\dR}$ is exact and symmetric monoidal, it commutes with finite truncations of the Chevalley-Eilenberg complex. Since the monoidal structures are pro-nilpotent by Lemma~\ref{lem:pro-nil} and~\cite{FG}, it commutes with the full Chevalley-Eilenberg functor.  
\end{proof}

\section{Moduli spaces of nodal curves} \label{sec:moduli}

\subsection{Stable weighted pointed curves}

Given $g,n \in \ZZ_{\geq 0}$, the moduli stack $\cM_{g,n}$ of $n$-pointed, smooth, projective, genus $g$ curves, admits a well-known compactification, given by the Deligne-Mumford stack $\bcM_{g,n}$: For a $k$-algebra $R$, its $R$-points are given by flat and proper families $X \to \Spec R$, together with $n$ sections $p_1,\dots,p_n : \Spec R \to X$, such that $X$ has at worst nodal singularities, namely points whose formal neighborhood is isomorphic to $k[[x,y]]/(xy)$, and such that the configuration $(X,p_1,\dots,p_n)$ is \textit{stable}. A pointed curve is stable if for each irreducible component $X_0$ of genus $g_0$, the number of nodes and marked points is greater than $2-2g_0$. In particular, $\bcM_{0,n}$ is empty for $n<3$, and $\bcM_{1,n}$ is empty for $n=0$.

However, there exist other compactifications of $\cM_{g,n}$. A family of such compactifications, which we will use in this paper, is given by Hassett's moduli space of weighted pointed curves~\cite{Has}:

\begin{definition}
  For a vector of rational weights $\pmb{w} \in (\QQ \cap [0,1])^n$, an $R$-family of stable pointed curves $(X,p_1,\dots,p_n)$ is $\pmb{w}$-stable if the $\QQ$-divisor $K + \sum_{i=1}^n w_i[p_i]$ is ample. Denote by $\bcM_{g,\pmb{w}}$ the moduli stack of $n$-pointed genus $g$ $\pmb{w}$-stable curves.
\end{definition}

Over geometric points, this means the following: The total weight at any given point is $\leq 1$, where a node has weight $1$, and the total weight of an irreducible component of genus $h$ is greater than $2 - 2h$. The stacks $\bcM_{g,\pmb{w}}$ have the usual stable reduction maps, as well as stable reduction maps $\bcM_{g,\pmb{w}} \to \bcM_{g,\pmb{w}'}$ whenever $w_i' \leq w_i$ for all $i$. It is proven in~\cite{Has} that these spaces are smooth and proper Deligne-Mumford stacks.

We will mostly be interested in a special case of this construction, where we only have two classes of marked points: Those with weight one, which behave like marked points on the usual compactification $\bcM_{g,n}$, and those with weight $\epsilon > 0$, which may collide with each other (but not with nodes or weight-one points), thus behave like Ran points.

\begin{definition}
    For $n\geq 0$ and a finite set $I$, let $\pmb{w}_{n,I} := (\pmb{1},\frac{\pmb{1}}{|I|+1}) \in \QQ^n \times \QQ^I$. Define
  \[\bcM_{g,n,I} := \bcM_{g,\pmb{w}_{n,I}}~.\]
\end{definition}

\begin{definition}
  For an $R$-point $(X,x_1,\dots,x_n) \in \bcM_{g,n}(R)$, let \[\bcM_{(X,x_1,\dots,x_n),I} = \bcM_{g,n,I} \times_{\bcM_{g,n}} \Spec R~.\]
\end{definition}

\begin{example}
  In the case $ (g,n) = (0,2)$, the stacks $\bcM_{0,2,I}$ are precisely the stacks considered by Losev-Manin in \cite{LM}. A $k$-point is given by a chain $\PP^1 \cup_{\infty \sim 0} \PP^1 \cup \dotsb \cup_{\infty \sim 0} \PP^1$ with $0$ of the first component and $\infty$ of the last one being the weight $1$ points, together with $I$ smooth points disjoint from $0,\infty$. As was described in~\cite[Chapter 6]{Has}, they can be obtained from $\PP^{|I|-1}$ by an explicit sequence of blow-ups and are therefore smooth, proper, and reduced varieties, of dimension $|I|-1$.
\end{example}

\begin{example} \label{ex:ss-modif-marked}
  Let $(X,x) \in \cM_{g,1}$ be a pointed smooth curve. Then $k$-points of the stack $\bcM_{(X,x),I}$ are given by curves of the form \[X' = X \cup_{x \sim 0} \PP^1 \cup \dotsb \cup_{\infty \sim 0} \PP^1\] with $\infty$ of the last component being the unique marked point of weight $1$, together with $I$ weight-$\epsilon$ marked points on the smooth locus, such that each genus $0$ component has at least one marked point of weight $\epsilon$. We call such $X'$ a \textit{semistable modification of $X$ at the marked point $x$} (this is called an expanded pair in~\cite{Li}). The variety $(X \backslash \{x\})^I$ embeds as the open substack of $\bcM_{(X,x),I}$ where the chain is trivial. 
\end{example}

\begin{example} \label{ex:ss-modif-node}
  Let $X \in \bcM_g$ be a nodal curve with a unique node $q \in X$. Then $(X \backslash \{q\})^I$ embeds as an open substack in $\bcM_{X,I}$, the latter being the stack of rational bridges inserted at the node. Its $k$-points are given by curves of the form \[X' = X \cup_{q \sim \{0,\infty\}} (\PP^1 \cup \dotsb \cup_{\infty \sim 0} \PP^1)\] together with $I$ weight-$\epsilon$ marked points on the smooth locus, such that each component of genus $0$ has at least one marked point of weight $\epsilon$. We call such $X'$ a \textit{semistable modification of $X$ at the node $q$} (this is called an expanded degeneration in~\cite{Li}).
\end{example}

A special feature of the stacks $\bcM_{g,n,I}$ is that, when forgetting the weight $\epsilon$ points, the result is a semistable curve:
\begin{lemma}
  Let $\fM_{g,n}^{\sst}$ be the moduli stack of $n$-pointed genus $g$ semistable curves. Then there is a map
  \[\bcM_{g,n,I} \to \fM_{g,n}^{\sst}\]
\end{lemma}

\begin{proof}
  Given a pointed curve $(X,x_1,\dots,x_n,\{p_i\}_{i \in I}) \in \bcM_{g,n,I}$, the corresponding $n$-pointed curve $(X,x_1,\dots,x_n)$ is prestable, and does not have a rational tail {--} any such tail has a total weight $\leq 1$, and since the total weight of $p_i,i \in I$ is $<1$, such a tail would have been unstable in $\bcM_{g,n,I}$ as well.
\end{proof}

We will mostly be interested in sheaves and spaces over $\bcM_{g,n,I}$ which have a flat connection with respect to changing the (weight $\epsilon$) marked points, but not with respect to changing the geometry of the curve (we regard a weight-one marked point as a puncture, so changing it would mean changing the geometry). Thus, we define: 
\begin{definition}
  Let \[(\bcM_{g,n,I})_{\dR/} = (\bcM_{g,n,I})_{\dR} \times_{(\fM_{g,n}^{\sst})_{\dR}} \fM_{g,n}^{\sst}\] be the relative de Rham stack. Its $R$-points are given by tuples $(X,x_1,\dots,x_n,\{y_i^{\red}\})$ such that $(X,x_1,\dots,x_n) \in \fM_{g,n}^{\sst}(R)$, and \[(X^{\red},x_1^{\red},\dots,x_n^{\red},\{y_i^{\red}\}) \in \bcM_{g,n,I}(R^{\red})~.\]
\end{definition}

For a surjective map of finite sets $\alpha : I \to J$, let
\[\Delta_{(\alpha)} : (\bcM_{g,n,J})_{\dR/} \to (\bcM_{g,n,I})_{\dR/}\]
be the map
\[(X,x_1,\dots,x_n,\{y_j^{\red}\}_{j \in J}) \mapsto (X,x_1,\dots,x_n,\{y_{\alpha(i)}^{\red}\}_{i \in I})~.\]

\begin{definition}
  Let \[\bcM_{g,n,\Ran} \coloneqq \colim_{\fSet^{\surj,\op}} (\bcM_{g,n,I})_{\dR/}~.\]
\end{definition}

By definition,
\[\IndCoh(\bcM_{g,n,\Ran}) = \lim_{I \in \fSet^{\surj}} \IndCoh((\bcM_{g,n,I})_{\dR/})~,\]
where the limit is taken with respect to $!$-pullbacks.

Since the maps $\Delta_{(\alpha)}$ are closed embeddings, $\Delta_{(\alpha)}^!$ admits a left adjoint $\Delta_{(\alpha),!} \simeq \Delta_{(\alpha),*}$, and we get:
\begin{corollary}
  \[\IndCoh(\bcM_{g,n,\Ran}) \simeq \lim_{I \in \fSet^{\surj}} \IndCoh((\bcM_{g,n,I})_{\dR/}) \iso \colim_{I \in \fSet^{\surj,\op}} \IndCoh((\bcM_{g,n,I})_{\dR/})~,\]
  where the colimit is taken with respect to $!$-pushforward. 
  
  In terms of this isomorphism, each $\cF \in \IndCoh(\bcM_{g,n,\Ran})$ can be written as
  \[\cF \simeq \colim \Delta_{(\alpha),!}\Delta_{(\alpha)}^!\cF~.\]
\end{corollary}

Assume now $2g+n-2>0$. Then we have a projection map
\[\bcM_{g,n,\Ran} \to \bcM_{g,n}\]
given by forgetting the weight-$\epsilon$ points, followed by stable reduction. Furthermore, this projection can be lifted to the Ran space of the universal curve:
\begin{definition} \label{def:Ran_gn}
  Let $\Ran_{g,n} \to \bcM_{g,n}$ be the relative Ran space, i.e., given by the colimit over finite sets and surjections of the stacks
  \[(\bcX_{g,n})^I_{\dR/} \times_{\bcM_{g,n}^I} \bcM_{g,n}\]
  where \[(\bcX_{g,n})_{\dR/} = (\bcX_{g,n})_{\dR} \times_{(\bcM_{g,n})_{\dR}} \bcM_{g,n}\] is the relative de Rham stack of the universal curve. Note that for $(X,x_1,\dots,x_n) \in \bcM_{g,n}(R)$, \[\Ran_{g,n} \times_{\bcM_{g,n}} \{(X,x_1,\dots,x_n)\} = \Ran(X/R)\] is the usual Ran space. Let $\mathring{\Ran}_{g,n} \subset \Ran_{g,n}$ be the open subspace corresponding to $\Ran(X^{\sm} \backslash \{x_1,\dots,x_n\}/R) \into \Ran (X/R)$.
\end{definition}

We can identify an object in the fiber of $\bcX_{g,n} \to \bcM_{g,n}$ as a marked point of weight $0$. In particular, we have a stable reduction map from points of weight $\epsilon$:
\begin{equation} \label{eq:Pi_gn}
  \Pi_{g,n} : \bcM_{g,n,\Ran} \to \Ran_{g,n}~.
\end{equation}

\begin{definition}
  For $(X,x_1,\dots,x_n) \in \bcM_{g,n}(R)$, let
  \[\bcM_{(X,x_1,\dots,x_n),\Ran} := \bcM_{g,n,\Ran} \times_{\bcM_{g,n}} \Spec R \simeq \colim_{\fSet^{\surj,\op}} (\bcM_{(X,x_1,\dots,x_n),I})_{\dR/}~.\]
\end{definition}

\begin{example} \label{ex:Pi_Xx}
  Following Example~\ref{ex:ss-modif-marked}, for a smooth pointed curve $(X,x)$, the space $\bcM_{(X,x),\Ran}$ classifies finite and reduced subsets of semistable modifications $X'$ of $X$ at $x$. This space will serve as a maximal compactification of $\Ran(X \backslash x)$. Restricting the map \eqref{eq:Pi_gn} to $(X,x)$, we get a map
  \[\Pi_{(X,x)} : \bcM_{(X,x),\Ran} \to \Ran X\]
  to the minimal compactification.
\end{example}

\begin{example}
  Following Example~\ref{ex:ss-modif-node}, for an $R$-family of nodal curves $X \in \bcM_g(R)$, the space $\bcM_{X,\Ran}$ classifies finite and reduced subsets of semistable modifications $X'$ of $X$ at the nodes. This space serves as a maximal compactification for the relative Ran space $\Ran(X^{\sm}/R)$, and by restricting~\eqref{eq:Pi_gn}, we get a map
  \[\Pi_{X/R} : \bcM_{X/R,\Ran} \to \Ran (X/R)\]
  to the minimal compactification.
\end{example}

Finally, we also have a semistable version of $\bcM_{g,n,\Ran}$:
\begin{definition}
    For $g,n \geq 0$, let \[\fM_{g,n,\Ran}^{\sst} \to \fM_{g,n}^{\sst}\] be the prestack whose fiber over $\Spec R \to \fM_{g,n}^{\sst}$ classifying a semistable $n$-pointed $R$-curve $(X,p_1,\dots,p_n)$ is given by \[\Ran_{\emptyset}(X^{\sm} \backslash \{p_1,\dots,p_n\}/R) \coloneqq \Ran(X^{\sm} \backslash \{p_1,\dots,p_n\}/R) \sqcup \{\emptyset_R\}~.\]

    For a finite (possibly empty) set $I$, let $\fM_{0,2,I}^{\sst} \to \fM_{0,2,\Ran}^{\sst}$ be the space of $I$-indexed finite subsets.
\end{definition}

In other words, we have \[\fM_{g,n,\Ran}^{\sst} = \Ran_{\emptyset}(\mathring{\fX}_{g,n}^{\sst}/\fM_{g,n}^{\sst})~,\]
where $\fX_{g,n}^{\sst} \to \fM_{g,n}^{\sst}$ is the universal curve, and $\mathring{\fX}_{g,n}^{\sst} = \fX_{g,n}^{\sst,\sm} \backslash \bigcup_{i=1}^n \sigma_i^*\fM_{g,n}^{\sst}$ is the smooth locus of the corresponding punctured curve. Here $\sigma_i : \fM_{g,n}^{\sst} \to \fX_{g,n}^{\sst}$ are the canonical sections. We then have a natural inclusion 
\begin{equation} \label{eq:Phi_gn}
    \Phi_{g,n} : \bcM_{g,n,\Ran} \to \fM_{g,n,\Ran}^{\sst}
\end{equation} 

and a left inverse over the non-empty locus, given by stable reduction
\begin{equation} \label{eq:Psi_gn}
    \Psi_{g,n} : \fM_{g,n,\Ran}^{\sst} \backslash \fM_{0,2,\emptyset}^{\sst} \to \bcM_{g,n,\Ran}~.
\end{equation}

\subsection{Moduli of semistable modifications for \texorpdfstring{$(g,n) = (0,2)$}{(g,n) = (0,2)}}

In this section we review some results about the moduli of semistable modifications in the case of two-pointed genus zero curves. The basic object of study will be the moduli stack $\fM_{0,2}^{\sst}$, parametrizing chains of genus zero curves (``rational chains''), with the two endpoints marked. A general semistable modification of an $n$-pointed genus $g$ curve, either at a marked point or at a node, is given by inserting a rational chain at that point, hence the importance of this special case.

A basic object in the description of this stack is the quotient stack $\AA^1/\GG_m$, parametrizing (virtual) Cartier divisors: A map $S \to \AA^1/\GG_m$ from a scheme $S$ is equivalent to a choice of a Cartier divisor $D \subset S$, given by pulling back the universal Cartier divisor $B\GG_m = \{0\}/\GG_m \into \AA^1/\GG_m$: \[D = S \times_{\AA^1/\GG_m} B\GG_m~.\]

Over $\AA^1$, we have the family \[\overline{\WW}[1] = \left\{([x_0:x_1],[y_0:y_1],t) \vcentcolon x_0y_0 = tx_1y_1\right\} \subset \PP^1 \times \PP^1 \times \AA^1 \xrightarrow{t} \AA^1\]
(or: the closure of $\{xy=t\} \subset \AA^3$ inside $\PP^1 \times \PP^1 \times \AA^1$). We have two sections \[0,\infty : \AA^1 \to \overline{\WW}[1]\] given by \[0,\infty : t \mapsto ([1:0],[0:1],t),([0:1],[1:0],t)\]respectively. This is a family of semistable, two-pointed, genus zero curves, hence defines a map \[\tilde{\tau}_1 : \AA^1 \to \fM_{0,2}^{\sst}~.\]

The curve $\overline{\WW}[1] \to \AA^1$ admits an action by $\GG^2_m$, given by \[(\lambda_0,\lambda_1) \cdot ([x_0:x_1],[y_0:y_1],t) = ([\lambda_0^{-1}x_0:x_1],[\lambda_1y_0:y_1],\lambda_0^{-1}t\lambda_1)\]
and so we get a map \[\tau_1 : \fA[1] \coloneqq \GG_m \backslash \AA^1 / \GG_m \to \fM_{0,2}^{\sst}~.\] As we will see, this map is an open embedding of the substack of curves with at most two irreducible components. More generally, Li~\cite{Li} defined families of semistable two-pointed genus zero curves \[\overline{\WW}[\ell] \to \AA^{\ell}\] for each ${\ell}>0$, having at most $\ell$ nodes (or: at most $\ell+1$ irreducible components). It has the property that the fiber over a point $(t_1,\dots,t_{\ell})$ with $i$ zeros has exactly $i$ nodes, and it can be equipped with a $\GG_m^{\ell+1}$ action.

\begin{definition}
    Write \[\fA[\ell] = \GG_m \backslash \AA^1 \times^{\GG_m} \dotsb \times^{\GG_m} \AA^1 = \AA^{\ell}/\GG_m^{\ell+1}\] For $\ell>0$, let \[\tau_{\ell} : \fA[\ell] \to \fM_{0,2}^{\sst}\] be the map classifying the family $\overline{\WW}[\ell] / \GG_m^{\ell+1} \to \fA[\ell]$. For $\ell=0$, let $\overline{\WW}[0] = \PP^1$, and \[\tau_0 : B\GG_m \to \fM_{0,2}^{\sst}\] the map classifying the family $\PP^1/\GG_m \to B\GG_m$.
\end{definition}

For a surjective order-preserving map $\alpha : [\ell] \onto [\ell']$, we have maps \[\fA[\ell'] \to \fA[\ell]\] given by \[(t_1,\dots,t_{\ell'}) \mapsto (t_{[\alpha(0)+1,\alpha(1)]},\dots,t_{[\alpha(\ell-1)+1,\alpha(\ell)]})~.\] 

The following is a combination of~\cite[Proposition~3.3.4,6.3.3,8.3.1]{ACFW}:
\begin{theorem} \label{thm:ss-02}
    The maps \[\tau_{\ell} : \fA[\ell] \to \fM_{0,2}^{\sst}\] are \'{e}tale covers of the open substack $\fM_{0,2}^{\sst,(\leq \ell+1)} \subset \fM_{0,2}^{\sst}$ of at most $\ell+1$ irreducible components, and we have an isomorphism \[\fM_{0,2}^{\sst} \simeq \colim_{[\ell] \in \Delta^{\surj,\op}} \fA[\ell]~.\]
\end{theorem}

\begin{remark}
    In~\cite{ACFW}, the colimit is taken along $\Delta^{\inj}$, and the $\ell$-th component corresponds to the locus of at most $\ell+2$ components. Those two presentations are equivalent, using the isomorphism $\Delta^{\inj}_+ \simeq \Delta^{\surj,\op}$, where $\Delta_+$ is the augmented simplex category (namely, with an additional object $[-1]$ and a unique map $[-1] \to [\ell]$), and the fact that the inclusion $\Delta^{\inj} \to \Delta^{\inj}_+$ is final.
\end{remark}

We will also need the following variants:
\begin{definition} \label{def:M2plus2}
    Let $\fM_{0,1+2}^{\sst}$ (resp. $\fM_{0,2+1}^{\sst}$) be the open substack of $\fM_{0,3}^{\sst}$ given by $3$-pointed curves $(X,x_0,x_1,x_\infty)$ such that $x_1,x_{\infty}$ (resp. $x_0,x_1$) are supported on the same irreducible component. Similarly, let $\fM_{0,2+2}^{\sst} \subset \fM_{0,4}^{\sst}$ be the open substack of $(X,x_0,x_1,x_{\infty-1},x_{\infty})$ such that each pair $x_0,x_1$ and $x_{\infty-1},x_{\infty}$ is supported on the same irreducible component. We will denote by $\fM_{0,2+1,\Ran}^{\sst}, \bcM_{0,2+1,\Ran}$ etc. the restriction of the Ran spaces to these open substacks.
\end{definition}

The following is essentially~\cite[Proposition~3.3.4]{ACFW}:
\begin{lemma} \label{lem:Gm-torsor}
    The projections \[\fM_{0,2+1}^{\sst},\fM_{0,1+2}^{\sst} \to \fM_{0,2}^{\sst}\] forgetting the point $x_1$ are $\GG_m$-torsors. Similarly, the projection \[\fM_{0,2+2}^{\sst} \to \fM_{0,2}^{\sst}\] forgetting the points $x_1,x_{\infty-1}$ is a $\GG_m^2$-torsor. Furthermore, we have $\GG_m$-equivariant identifications \[\fM_{0,2+2}^{\sst} / \GG_m \simeq \fM_{0,2+1}^{\sst}\] and \[\GG_m \backslash \fM_{0,2+2}^{\sst} \simeq \fM_{0,1+2}^{\sst}~.\]
\end{lemma}

\begin{proof}
    Given $(X,x_0,x_\infty) \in \fM_{0,2}^{\sst}(R)$, a lift to $\fM_{0,2+1}^{\sst}$ (resp. $\fM_{0,1+2}^{\sst}$, resp. $\fM_{0,2+2}^{\sst}$) is equivalent to a trivialization of the conormal bundle at $x_{\infty}$ (resp. $x_0$, resp. $x_0$ and $x_\infty$), and thus is a $\GG_m$ (resp. $\GG_m^2$) torsor.
\end{proof}

\begin{lemma}
    There exists a projection map
    \begin{equation} \label{eq:proj-ss02}
    \fM_{0,2}^{\sst} \to \fM_{0,2}^{\sst,(\leq 2)} \simeq \fA[1]
    \end{equation}
    which lifts to a projection of the universal curve \[\fX_{0,2}^{\sst} \to \fX_{0,2}^{\sst,(\leq 2)}~.\] 
\end{lemma}

\begin{proof}
    The projection follows from the description in Theorem~\ref{thm:ss-02}: For each $\ell$, we have multiplication maps $\fA[\ell] \to \fA[1]$ induced by the multiplicative monoid $\AA^1$. Since this map commutes with the transition maps $\fA[\ell] \to \fA[\ell']$, it defines a map from the colimit. The lift to the universal curve is given by taking the $\GG_m^2$ quotient of the stable reduction map \[\begin{tikzcd}
	{\fX_{0,2+2}^{\sst}} & {\bcX_{0,2+2}} \\
	{\fM_{0,2+2}^{\sst}} & {\bcM_{0,2+2} \simeq \AA^1}
	\arrow[from=1-1, to=1-2]
	\arrow[from=1-1, to=2-1]
	\arrow[from=1-2, to=2-2]
	\arrow[from=2-1, to=2-2]
    \end{tikzcd}\]  Alternatively, it can be obtained using the linear system $\cO(x_0+x_\infty)$, Cf.~\cite[Lemma~3.3.3]{ACFW}.
\end{proof}

In particular, we get projection maps \[\fM_{0,2+1}^{\sst}, \fM_{0,1+2}^{\sst} \to \AA^1/\GG_m\] and \[\fM_{0,2+2}^{\sst} \to \AA^1\] which map the smooth locus to $\GG_m$ and the singular locus to $\{0\}$. 

\begin{definition} \label{def:M2plus2completed}
    Let $\bcM_{0,2+1,\Ran}^{\land}$ (resp. $\bcM_{0,1+2,\Ran}^{\land}$, resp. $\bcM_{0,2+2,\Ran}^{\land}$) be the formal completion along the closed embedding of $\bcM_{0,2,\Ran}$ given by $(\PP^1,0,1,\infty) \lor (-)$ (resp. $(-) \lor (\PP^1,0,1,\infty)$, resp. $(\PP^1,0,1,\infty) \lor (-) \lor (\PP^1,0,1,\infty)$).
\end{definition}

\subsection{Moduli of semistable modifications for general \texorpdfstring{$(g,n)$}{(g, n)}}

In this section, we describe the structure of semistable modifications at a point (either smooth or nodal) in terms of the stack $\fM_{0,2}^{\sst}$ and its variants above. Informally, a semistable modification of a curve $X$ at a smooth point $x$ is given by attaching a (possibly empty) rational tail at $x$. Since this construction is local, it can be identified with attaching a rational tail to $\PP^1$ at $\infty$. This data is equivalent to a choice of an object $X \in \fM_{0,2}^{\sst}(k)$, together with an identification of the component of $0$ with $\PP^1$. Similarly, a semistable modification at a node is given by inserting a rational bridge, which is equivalent to the insertion of a rational bridge at $\infty \sim 0' \in \PP^1 \cup_{\infty \sim 0'} \PP^1$. The latter is equivalent, up to identifying the components of $0,\infty$ with $\PP^1$, with an object of $\fM_{0,2}^{\sst}$. The precise statement was described in~\cite{ACFW} (under the names ``expanded pairs'' and ``expanded degenerations''). We will follow their description, and explain the corresponding structure on Ran spaces.

\begin{proposition} \label{prop:structure_MXxRan_smooth}
    For a smooth pointed curve $(X,x) \in \cM_{g,1}(k)$, there is an isomorphism \[\fM_{(X,x)}^{\sst} \simeq \fM_{0,2+1}^{\sst}\] unique up to a choice of a local coordinate at $x$. Under this identification, we have a Cartesian square 
    \[\begin{tikzcd}
	{\bcM_{0,2+1,\Ran}^{\land}} & {\bcM_{(X,x),\Ran}} \\
	{\{x\}} & {\Ran X}
	\arrow[from=1-1, to=1-2]
	\arrow[from=1-1, to=2-1]
	\arrow[from=1-2, to=2-2]
	\arrow[from=2-1, to=2-2]
    \end{tikzcd}\]
\end{proposition}

\begin{proof}
    The isomorphism $\fM_{(X,x)}^{\sst} \simeq \fM_{0,2+1}^{\sst}$ is~\cite[1.3.2(2)]{ACFW}. Restricted to reduced algebras $R^{\red}$, it can be described explicitly as follows: Define a map \[\fM_{(X,x)}^{\sst}(R^{\red}) \to \fM_{0,2+1}^{\sst}(R^{\red})\] by joining $(\PP^1_{R^{\red}},0,1,\infty)$ to the exceptional divisor of an $R^{\red}$-point $(X'^{\red},x'^{\red}) \in \fM_{(X,x)}^{\sst}(R^{\red})$: \[(X'^{\red},x'^{\red}) \mapsto (\PP^1_{R^{\red}} \cup_{\infty \sim x_0^{\red}} (X'^{\red} \times_{X} \{x\}), 0,1,x'^{\red}) \eqqcolon (Y^{\red},y_0^{\red},y_1^{\red},y_{\infty}^{\red})\] where $x_0^{\red} = \overline{(X \backslash \{x\} \times_{X} X'^{\red})} \cap (\{x\} \times_{X} X'^{\red})$. Its inverse is given in a similar way: \[(Y^{\red},y_0^{\red},y_1^{\red},y_\infty^{\red}) \mapsto (X^{\red} \cup_{x^{\red} \sim y'^{\red}} (Y^{\red} \times_{\PP^1_{\red}} \{0\}),y_0^{\red})\] where $y'^{\red} = \overline{(\PP^1_{R^{\red}} \backslash \{0\}) \times_{\PP^1_{R^{\red}}} Y^{\red}} \cap (\{0\} \times_{\PP^1_{R^{\red}}} Y^{\red})$. We see that over the reduced locus, the isomorphism is unambiguously determined.
    
    A general $R$-point of $\fM_{(X,x)}^{\sst}$ is given by an $R$-deformation $(X',x')$ of $(X'^{\red},x'^{\red})$. By the above construction, we can identify it with an $R$-deformation $(Y,y_0,y_1,y_{\infty})$ of $(Y^{\red},y_0^{\red},y_1^{\red},y_\infty^{\red})$, up to a choice of a local coordinate at $x_0^{\red}$, which can then be used to identify deformations at $x_0^{\red}$ with deformations at $y'^{\red}$. 

    For the fiber square, note that an $R$-point of the fiber product \[\{x\} \times_{\Ran X} \bcM_{(X,x),\Ran}\] is given by a semistable modification $(X',x') \in \fM_{(X,x)}^{\sst}(R)$, together with a finite subset $\{x_i^{\red} : \Spec R^{\red} \to X'\}$ which is supported on the exceptional divisor of $X' \to X$ and away from $x_0$, and intersects nontrivially with each irreducible component of $X'^{\red}$. In particular, $X'^{\red}$ is strictly nodal, and so $(X',x')$ is contained in the formal neighborhood of the image of \[(\PP^1,0,\infty) \lor (-) : \fM_{0,2}^{\sst} \to \fM_{0,2+1}^{\sst}\] Such a subset is equivalent to a finite subset of $Y^{\red} \coloneqq \PP^1_{\red} \cup_{\infty \sim x_0^{\red}} (X'^{\red} \times_{X} \{x\})$ as above, which is supported on the second component \[\{y_i^{\red} : \Spec R^{\red} \xrightarrow{x_i^{\red}} X'^{\red}\times_X \{x\} \to \PP^1_{R^{\red}} \cup_{\infty \sim x_0} (X'^{\red} \times_X \{x\})\}\] Let now $(Y,y_0,y_1,y_\infty)$ be the $R$-deformation corresponding to $(Y^{\red},0,1,x'^{\red})$. Then we get precisely an $R$-point $(Y,y_0,y_1,y_\infty,\{y_i^{\red}\}) \in \bcM_{0,2+1,\Ran}^{\land}(R)$.
\end{proof}

\begin{definition} \label{def:nodal-deg}
    A \textit{nodal degeneration family} over a pair $(B,b)$ of a smooth curve $B$ and a closed point $b \in B$ is a flat family of proper curves $X/B$, such that the restriction $X_{\mathring{B}}$ to $\mathring{B} \coloneqq B \backslash {b}$ is smooth, and the fiber $X_b$ is nodal with a single node $p \in X_b$.
\end{definition}

\begin{definition}
    Let $\bcM_{0,2+2,\Ran}^{\land,\nd} \subset \bcM_{0,2+2,\Ran}^{\land}$ be the closed substack of curves which are either of the form $(\PP^1,0,1,\infty) \lor (X,x_0,x_\infty,\{x_i^{\red}\})$ or $(X,x_0,x_\infty,\{x_i^{\red}\}) \lor (\PP^1,0,1,\infty)$ for some $(X,x_0,x_\infty,\{x_i^{\red}\}) \in \bcM_{0,2,\Ran}$. Namely \[\bcM_{0,2+2,\Ran}^{\land,\nd} \simeq \colim_{I \in \fSet^{\surj,\op}}\bcM_{0,2+1,I}^{\land} \underset{\bcM_{0,2,I}}{\cup} \bcM_{0,1+2,I}^{\land}~,\] where the pushout is taken in the category of stacks and the colimit over finite sets is taken in the category of prestacks (for the existence of this pushout see~\cite[Appendix~A]{Hall}).
\end{definition}

\begin{proposition} \label{prop:structure_MXRan_nodal}
    Let $X/B$ be a nodal degeneration as in Definition~\ref{def:nodal-deg}. 
    \begin{enumerate}
        \item There is an equivalence \[\fM_{X/B}^{\sst} \simeq \fM_{0,2+2}^{\sst}/\GG_m \times_{\AA^1/\GG_m} B\] where the $\GG_m$-action corresponds to the diagonal action of $\GG_m$ coming from Lemma~\ref{lem:Gm-torsor}, the map $\fM_{0,2+2}/\GG_m \to \AA^1/\GG_m$ is the one coming from the projection~\eqref{eq:proj-ss02}, and the map $B \to \AA^1/\GG_m$ is the map classifying the divisor $b$. 
        \item We have a Cartesian square 
        \[\begin{tikzcd}
	    {\bcM_{0,2+2,\Ran}^{\land,\nd}} & {\bcM_{X/B,\Ran}} \\
	    {\{p\}/\{b\}} & {\Ran (X/B)}
	    \arrow[from=1-1, to=1-2]
	    \arrow[from=1-1, to=2-1]
	    \arrow[from=1-2, to=2-2]
	    \arrow[from=2-1, to=2-2]
        \end{tikzcd}\]
        \item We have a Cartesian square \[\begin{tikzcd}
        {\bcM_{0,2+2,\Ran}^{\land}} & {\bcM_{X/B,\Ran}} \\
        {(X_{B\mh\dR})^{\land}_{p}} & {\Ran (X/B)}
        \arrow[from=1-1, to=1-2]
        \arrow[from=1-1, to=2-1]
        \arrow[from=1-2, to=2-2]
        \arrow[from=2-1, to=2-2]
        \end{tikzcd}\]
        \end{enumerate}
\end{proposition}

\begin{proof}
    For (1), the isomorphism $\fM_{X/B}^{\sst} \simeq \fM_{0,2+2}^{\sst}/\GG_m \times_{\AA^1/\GG_m} B$ is a special case of~\cite[Theorem~1.3.2(3)]{ACFW}. When restricting to $X_{b}^{\land} \coloneqq X \times_{B_{\dR}} \{b\}$, this isomorphism becomes \begin{equation} \label{eq:deform-X_b}
        \fM_{X_b^{\land}/\hat{D}_b}^{\sst} \simeq \fM_{0,2+2}^{\sst}/\GG_m \times_{\AA^1/\GG_m} \hat{D}_b \simeq \fM_{0,2+2}^{\sst} \times_{\AA^1} \hat{D}_0~.
    \end{equation} Given a $k$-algebra $R$, the $R^{\red}$-points of the isomorphism can be described explicitly as follows: For $X'^{\red} \in \fM_{X_b}^{\sst}(R^{\red})$, let $X'^{\red}_0 = X'^{\red} \times_{X_b} \{p\}$ be the exceptional divisor, and $\{x_0^{\red},x^{\red}_\infty\} = X'^{\red}_0 \cap \overline{X'^{\red} \times_{X_b} (X_b \backslash \{p\})}$ (we might have $x_0^{\red} = x_\infty^{\red}$). This defines \[(Y^{\red},0_1,1_1,1_2,\infty_2) \coloneqq (\PP^1_{R^{\red}} \underset{\infty_1 \sim x_0}{\cup} X'^{\red}_0 \underset{x_\infty^{\red} \sim 0_2}{\cup} \PP^1_{R^{\red}}, 0_1,1_1,1_2,\infty_2) \in (\fM_{0,2+2}^{\sst} \times_{\AA^1} \{0\})(R^{\red})~.\] The construction in the other direction is similar. Now an $R^{\red}$-point of the fiber product \[\{p\}/\{b\} \times_{\Ran(X/B)} \bcM_{X/B,\Ran}\] is given by $X' \in \fM_{X_b}^{\sst}(R^{\red})$, together with a finite subset $\{x_i^{\red} : \Spec R^{\red} \to X'_{0}\}$ supported on the smooth locus, and in particular away from $x_0,x_\infty$. Thus, the compositions $y_i^{\red}$ of $x_i^{\red}$ with the map $X'^{\red}_0 \to Y^{\red}$ define \[(Y^{\red},0_1,1_1,1_2,\infty_2,\{y_i^{\red}\}) \in \bcM_{0,2+2,\Ran}^{\land}(R^{\red})\] Given an $R^{\red}$-point as above, a lift to an $R$-point of $\bcM_{X_b^{\land}/\hat{D}_b,\Ran}$ is given by an $R$-deformation $X' \in \fM_{X^{\land}_b/\hat{D}_b}^{\sst}(R)$. By~\eqref{eq:deform-X_b}, such a deformation is exactly a deformation of $Y^{\red}$, which proves (3). Finally, such a deformation maps to $\{p\} \subset (X_{B\mh\dR})_p^{\land}$ precisely if $X'$, and therefore $Y$, is strictly nodal. This is given by the subspace \[\bcM_{0,2+2,\Ran}^{\land,\nd} \subset \bcM_{0,2+2,\Ran}^{\land}~,\] which proves (2).
\end{proof}

\section{Chiral algebras on semistable modifications} \label{sec:chial-nodal}

\subsection{Chiral action on stable curves}

In this section we discuss the chiral monoidal structure, defined in Section~\ref{sec:ch}, in the context of families of stable and semistable curves. First note that, since $\fM_{g,n,\Ran}^{\sst} \to \fM_{g,n}^{\sst}$ is a family of smooth curves over the stack $\fM_{g,n}^{\sst}$, we have a chiral product \[\fM_{g,n,\Ran}^{\sst} \times_{\fM_{g,n}^{\sst}} \fM_{g,n,\Ran}^{\sst} \xleftarrow{\jmath} (\fM_{g,n,\Ran}^{\sst} \times_{\fM_{g,n}^{\sst}} \fM_{g,n,\Ran}^{\sst})_{\disj} \xrightarrow{\sqcup} \fM_{g,n,\Ran}^{\sst}~.\]

Similarly, $\mathring{\Ran}_{g,n} \to \bcM_{g,n}$ (see Definition~\ref{def:Ran_gn}) is a family of smooth curves over the stack $\bcM_{g,n}$, and so we have a chiral multiplication \[\mathring{\Ran}_{g,n} \times_{\bcM_{g,n}} \mathring{\Ran}_{g,n} \xleftarrow{\jmath} (\mathring{\Ran}_{g,n} \times_{\bcM_{g,n}} \mathring{\Ran}_{g,n})_{\disj} \xrightarrow{\sqcup} \mathring{\Ran}_{g,n}~.\]  

Given a universal factorization algebra $\cA$, our goal in this section is to construct a factorization module $\Vac(\cA)_{g,n}^0$ defined over nodal curves $\Ran_{g,n}$, with respect to the action of $\cA_{g,n}^0$, defined over the smooth, punctured locus $\mathring{\Ran}_{g,n}$.

\begin{definition}
    Let \[\Ran_{g,n}^{\sst} = \Ran_{g,n} \times_{\bcM_{g,n},\Psi_{g,n}} \fM_{g,n}^{\sst}\] be the base-change of $\Ran_{g,n}$ to a family of curves over $\fM_{g,n}^{\sst}$. Here $\Psi_{g,n}$ is the stable reduction map of \eqref{eq:Psi_gn}. Define $\mathring{\Ran_{g,n}}^{\sst} \subset \Ran_{g,n}^{\sst}$ by \[\mathring{\Ran_{g,n}}^{\sst} = \Ran_{g,n}^{\sst} \times_{\Ran_{g,n}} \mathring{\Ran}_{g,n}~.\]
\end{definition}

Explicitly, an $R$-point of $\Ran_{g,n}^{\sst}$ is given by an $n$-pointed semistable curve $(X,\overline{p})$, and a finite and reduced subset $\{x_i^{\red}\}_{i \in I} \subset X^{\operatorname{st}}(R^{\red})$ of its stable reduction. Denote by $\operatorname{St} : X \to X^{\operatorname{st}}$ the projection. Then a point as above lies within $\mathring{\Ran}_{g,n}^{\sst}$ if the image of $\{x_i^{\red}\}_{i \in I}$ is contained in $X^{\operatorname{st,sm}} \backslash \{\operatorname{St}(p_1),\dots,\operatorname{St}(p_n)\}$.

\begin{lemma}
    The chiral multiplication on $\mathring{\Ran}_{g,n}$ is compatible with base-change along the stable reduction map $\Psi_{g,n}$ in the sense of Definition~\ref{def:Ran-compat-bc}.
\end{lemma}

\begin{proof}
    The fiber product \[(\mathring{\Ran}_{g,n}^{\sst} \times_{\fM_{g,n}^{\sst}} \mathring{\Ran}_{g,n}^{\sst}) \times_{(\mathring{\Ran}_{g,n}^{\sst} \times_{\bcM_{g,n}} \mathring{\Ran}_{g,n})} (\mathring{\Ran}_{g,n} \times_{\bcM_{g,n}} \mathring{\Ran}_{g,n})_{\disj}\] classifies pairs of semistable curves, together with finite, reduced, and disjoint subsets of their stable reduction, which are supported on the smooth locus and away from the weight-$1$ marked points. Since the stable reduction map is an isomorphism over the smooth unmarked locus, this corresponds exactly to a pair of disjoint subsets of the original curves, supported away from the exceptional divisor. 

\end{proof}

Therefore, the chiral monoidal structure on the family $\mathring{\Ran}_{g,n}^{\sst}$ is obtained by base-change from the chiral multiplication on $\mathring{\Ran}_{g,n}$. By Lemma~\ref{lem:base-change-fact-mod}, $\Ran_{g,n}^{\sst} \in \mathring{\Ran}_{g,n}^{\sst}\mh\Mod^{\et}$ is a factorization module space over it. 

\begin{lemma} \label{lem:st-mod}
    The open embedding $\Phi_{g,n} : \bcM_{g,n,\Ran} \to \fM_{g,n,\Ran}^{\sst}$ of \eqref{eq:Phi_gn} upgrades to an augmentation of $\fM_{g,n,\Ran}^{\sst}$-modules.
\end{lemma}

\begin{proof}
    The action and augmentation are defined via the diagram
    \[\begin{tikzcd}
	{\bcM_{g,n,\Ran} \times_{\fM_{g,n}^{\sst}} \fM_{g,n,\Ran}^{\sst,I}} & {(\bcM_{g,n,\Ran} \times_{\fM_{g,n}^{\sst}} \fM_{g,n,\Ran}^{\sst,I/})_{\disj}} & {\bcM_{g,n,\Ran}} \\
	{\fM_{g,n,\Ran}^{\sst,I_+}} & {(\fM_{g,n,\Ran}^{\sst})^{I_+/}_{\disj}} & {\fM_{g,n,\Ran}^{\sst}}
	\arrow["{\Phi_{g,n} \times \id}", from=1-1, to=2-1]
	\arrow["{\jmath_{\st}}"', from=1-2, to=1-1]
	\arrow["{\sqcup_{\st}}", from=1-2, to=1-3]
	\arrow["{\Phi_{g,n,I}}", from=1-2, to=2-2]
	\arrow["{\Phi_{g,n}}", from=1-3, to=2-3]
	\arrow["{\jmath}"', from=2-2, to=2-1]
	\arrow["{\sqcup}", from=2-2, to=2-3]
    \end{tikzcd}\] where $I_+ = I \sqcup \{+\}$ and the term $(\bcM_{g,n,\Ran} \times_{\fM_{g,n}^{\sst}} \fM_{g,n,\Ran}^{\sst,I/})_{\disj}$ is defined as the pullback along the left square. $\sqcup_{\st}$ is well-defined since the union of a stable and a semistable configuration of weight-$\epsilon$ points is again stable.
\end{proof}

\begin{definition}
    Let \begin{align*}
        \Fact_{g,n}^{0,\sst} & = \Fact(\mathring{\Ran}_{g,n}^{\sst}/\fM_{g,n}^{\sst}) \\
        \Fact_{g,n}^{\epsilon,\sst} & = \Fact(\fM_{g,n,\Ran}^{\sst}/\fM_{g,n}^{\sst}) \\
        \Fact_{g,n}^0 & = \Fact(\mathring{\Ran}_{g,n}/\bcM_{g,n})~.
    \end{align*}
\end{definition}

Define a map \[\Upsilon_{g,n} : \mathring{\Ran}_{g,n}^{\sst} \to \fM_{g,n,\Ran}^{\sst}\] by \[(X,p_1,\dots,p_n,\{x_i^{\red}\}_{i \in I}) \mapsto \left(X,p_1,\dots,p_n,\left\{\operatorname{St}^{-1}(x_i^{\red})\right\}_{i\in I}\right)~.\]
Note that this is well-defined, since $\operatorname{St}$ is an isomorphism away from the exceptional divisor. From Lemma~\ref{lem:open-fact} and Lemma~\ref{lem:open-fact-mod}, with $S = \fM_{g,n}^{\sst}$, $\Ran(X/S) = \fM^{\sst}_{g,n,\Ran}$ and $\Ran(U/S) = \mathring{\Ran}_{g,n}^{\sst}$, we get:
\begin{corollary}
    $\Upsilon_{g,n}$ induces a map \[\Upsilon_{g,n}^{\circ} : \fM_{g,n,\Ran}^{\sst}\mh\Mod^{\et} \to \mathring{\Ran}_{g,n}^{\sst}\mh\Mod^{\et}\] as well as a map \[\Upsilon_{g,n}^{!,\fact} : \Fact_{g,n}^{\epsilon,\sst} \to \Fact_{g,n}^{0,\sst}\] Combining with Lemma~\ref{lem:st-mod}, we get a $\mathring{\Ran}^{\sst}_{g,n}$-module structure on $\bcM_{g,n,\Ran}$.
\end{corollary}

From Theorem~\ref{thm:loc-fact}, we get for each universal factorization algebra $\cA$ a family of factorization algebras over these spaces.
\begin{definition} \label{def:univ-to-Mgn}
    For a universal factorization algebra $\cA$ and $g,n \geq 0$, denote by $\cA_{g,n}^{\epsilon,\sst} \in \Fact_{g,n}^{\epsilon,\sst}$ the corresponding family of factorization algebras. Let \[\cA_{g,n}^{0,\sst} = \Upsilon^{!,\fact}_{g,n}\cA_{g,n}^{\epsilon,\sst}~.\]
\end{definition}

\begin{definition}
    Let \begin{align*}
        \cA_{g,n}^{\epsilon,\sst}\mh\FactMod^{\epsilon} & = \cA_{g,n}^{\epsilon,\sst}\mh\FactMod(\bcM_{g,n,\Ran}) \\
        \cA_{g,n}^{0,\sst}\mh\FactMod^{\epsilon} & = \cA_{g,n}^{0,\sst}\mh\FactMod(\bcM_{g,n,\Ran}) \\
        \cA_{g,n}^{0,\sst}\mh\FactMod^{0} & = \cA_{g,n}^{0,\sst}\mh\FactMod(\Ran_{g,n}^{\sst})~.
    \end{align*} 
\end{definition}

\begin{remark}
    $\cA_{g,n}^{0,\sst}$ can be equivalently defined as the factorization algebra attached to the $\fM_{g,n}^{\sst}$-family of curves $\mathring{\Ran}_{g,n}^{\sst}$.
\end{remark}

Now from the map $\Pi_{g,n}$ of \eqref{eq:Pi_gn} and the map $\bcM_{g,n,\Ran} \to \fM_{g,n}^{\sst}$, we get a map \[\tilde{\Pi}_{g,n} : \bcM_{g,n,\Ran} \to \Ran_{g,n}^{\sst}\]
Explicitly, for an $R$-point $(X,x_1,\dots,x_n,\{y_i^{\red}\}_{i \in I}) \in \bcM_{g,n,\Ran}(R)$, let \[\operatorname{St} : (X,x_1,\dots,x_n) \to (X^{\operatorname{st}},x'_1,\dots,x'_n)\] be the stable reduction map. Then \[\tilde{\Pi}_{g,n} : \left(X,x_1,\dots,x_n,\left\{y_i^{\red}\right\}_{i \in I}\right) \mapsto \left(X,x_1,\dots,x_n,\left\{\operatorname{St}(y_i^{\red})\right\}_{i \in I}\right)~.\]

\begin{lemma} \label{lem:Pi-bimod-map}
    $\tilde{\Pi}_{g,n}$ is both a map of $\mathring{\Ran}_{g,n}^{\sst}$-modules and of $\mathring{\Ran}_{g,n}^{\sst}$-comodules.
\end{lemma}

\begin{proof}
    For each finite set $I$, we have a commutative diagram \[\begin{tikzcd}
	{\bcM_{g,n,\Ran} \times_{\fM_{g,n}^{\sst}} \mathring{\Ran}_{g,n}^{\sst,I}} & {(\bcM_{g,n,\Ran} \times_{\fM_{g,n}^{\sst}} \mathring{\Ran}_{g,n}^{\sst,I/})_{\disj}} & {\bcM_{g,n,\Ran}} \\
	{\Ran_{g,n}^{\sst} \times_{\fM_{g,n}^{\sst}} \mathring{\Ran}_{g,n}^{\sst,I}} & {(\Ran_{g,n}^{\sst} \times_{\fM_{g,n}^{\sst}} \mathring{\Ran}_{g,n}^{\sst,I/})_{\disj}} & {\Ran_{g,n}^{\sst}}
	\arrow["{\tilde{\Pi}_{g,n} \times \id}", from=1-1, to=2-1]
	\arrow["{\jmath_{\epsilon}}"', from=1-2, to=1-1]
	\arrow["{\sqcup_{\epsilon}}", from=1-2, to=1-3]
	\arrow["{\tilde{\Pi}_{g,n,I}}", from=1-2, to=2-2]
	\arrow["{\tilde{\Pi}_{g,n}}", from=1-3, to=2-3]
	\arrow["{\jmath_0}"', from=2-2, to=2-1]
	\arrow["{\sqcup_0}", from=2-2, to=2-3]
    \end{tikzcd}\] where $\tilde{\Pi}_{g,n,I}$ is well-defined since any $\mathring{\Ran}_{g,n}^{\sst}$ marked point is supported on the stable locus, over which the stable reduction map is an isomorphism, and so $\tilde{\Pi}_{g,n}$ preserves disjointness. Therefore, since both squares are Cartesian if we remove the disjoint condition, they remain so after we add it.
\end{proof}

Combining Lemma~\ref{lem:Pi-bimod-map} with Lemma~\ref{lem:module-map-fact}, we get:
\begin{corollary} \label{cor:eps-to-zero-mod}
    The map $\tilde{\Pi}_{g,n}$ induces a map \[(\tilde{\Pi}_{g,n})_* : \cA_{g,n}^{0,\sst}\mh\FactMod^{\epsilon} \to \cA_{g,n}^{0,\sst}\mh\FactMod^{0}~.\]
\end{corollary}

\begin{example} \label{ex:vac-mod-gn}
    As in Example~\ref{ex:vac-mod}, we have the vacuum module \[\Vac(\cA)_{g,n}^{\epsilon} \coloneqq \Phi_{g,n}^!\cA_{g,n}^{\epsilon,\sst} \in \cA_{g,n}^{\epsilon,\sst}\mh\FactMod^{\epsilon}~.\]
    Using Corollary~\ref{cor:eps-to-zero-mod}, we get \[\Vac(\cA)_{g,n}^{0,\sst} \coloneqq (\tilde{\Pi}_{g,n})_*\Vac(\cA)_{g,n}^{\epsilon} \in \cA_{g,n}^{0,\sst}\mh\FactMod^{0}\] This defines an $\cA$-vacuum module for $\Ran_{g,n}^{\sst}$ over $\mathring{\Ran}_{g,n}^{\sst}$, despite the fact that the latter is not an augmented module.
\end{example}

Finally, using Lemma~\ref{lem:base-change-fact} and Lemma~\ref{lem:base-change-fact-mod}, with $\Psi_{g,n} : T = \fM^{\sst}_{g,n} \to \bcM_{g,n} = S$ (see \eqref{eq:Psi_gn}) and $\Ran(X/S) = \mathring{\Ran}_{g,n}$, we get:
\begin{corollary}\label{cor:vac-0}
    The pushforward of $\cA_{g,n}^{0,\sst}$ along the unit map $u_{\Psi} : \mathring{\Ran}^{\sst}_{g,n} \to \mathring{\Ran}_{g,n}$ defines a factorization algebra $u_{\Psi,*}\cA_{g,n}^{0,\sst} \eqqcolon \cA_{g,n}^0 \in \Fact_{g,n}^0$, and we have a map
    \[u_{\Psi,*}^{\cA} : \cA_{g,n}^{0,\sst}\mh\FactMod^{0} \to \cA_{g,n}^0\mh\FactMod^0 \coloneqq \cA_{g,n}^{0}\mh\FactMod(\Ran_{g,n})\] We denote \[\Vac(\cA)^0_{g,n} \coloneqq u_{\Psi,*}^{\cA}\Vac(\cA)_{g,n}^{0,\sst}~.\]
\end{corollary}

The module $\Vac(\cA)^0_{g,n}$ will be our extension of a universal factorization algebra to nodal curves. 

\begin{definition} \label{def:VacXx}
    For $f : \Spec R \to \bcM_{g,n}$, let $f_{\Ran}^{\epsilon} : \bcM_{(X,x_1,\dots,x_n),\Ran} \to \bcM_{g,n,\Ran}$ and $f_{\Ran}^0 : \Ran(X/R) \to \Ran_{g,n}$ be its base-change. Define \begin{align*}
        \Vac(\cA)_{(X,x_1,\dots,x_n)}^{\epsilon} & = f_{\Ran}^{\epsilon,!}\Vac(\cA)_{g,n}^{\epsilon} \\
        \Vac(\cA)_{(X,x_1,\dots,x_n)}^0 & = f_{\Ran}^{0,!}\Vac(\cA)_{g,n}^{0}~.
    \end{align*}
\end{definition}

Note that by base-change, we have \[\Vac(\cA)_{(X,x_1,\dots,x_n)}^0 = \Pi_{(X,x_1,\dots,x_n),!}\Vac(\cA)_{(X,x_1,\dots,x_n)}^{\epsilon}\] where $\Pi_{(X,x_1,\dots,x_n)}$ is as in Example~\ref{ex:Pi_Xx}.

\subsection{The join monoidal structure}

In this section, we will see that for $(g,n) = (0,2)$, the (pre)stacks of semistable modifications (namely $\fM_{0,2}^{\sst},\fM_{0,2,\Ran}^{\sst},\bcM_{0,2,\Ran}$) admit an additional structure of an associative algebra, given by concatenation of rational chains, and a universal factorization algebra produces an equivariant sheaf with respect to this algebra structure. From that we will get an associative algebra structure on global sections, which, as we prove later, is the value at a node of the extension $\Vac(\cA)^0_{g,n}$ of Corollary~\ref{cor:vac-0}.

\begin{construction} \label{con-join}
    For $\ell >0$ and $(X^i,x^i_0,x^i_\infty) \in \fM_{0,2}^{\sst};\: 1 \leq i \leq \ell$, define \[\lor_{1\leq i \leq \ell} (X^i,x^i_0,x^i_\infty) = (X^1 \cup_{x_\infty^1 \sim x_0^2} X^2 \cup \dotsb \cup_{x_\infty^{\ell-1} \sim x_0^{\ell}}X^{\ell}, x_0^1,x_\infty^{\ell}) ~.\]

    For $(X^i,x^i_0,x^i_\infty,\{y_{i,j}^{\red}\}_{j \in J_i}) \in \fM^{\sst}_{0,2,\Ran};\:1 \leq i\leq \ell$, define \[\lor_{\Ran,1\leq i \leq \ell} (X^i,x^i_0,x^i_\infty, \{y_{i,j}^{\red}\}_{j \in J_i}) = (\lor_{1\leq i \leq \ell} (X^i,x^i_0,x^i_\infty), \cup_{1 \leq i \leq \ell}\{y_{i,j}^{\red}\}_{j \in J_i}) ~.\]
\end{construction}

This defines a non-unital associative algebra structure on $\fM_{0,2}^{\sst},\fM_{0,2,\Ran}^{\sst}$. Furthermore, the projection commutes with the multiplication maps, and we have:
\begin{lemma} \label{lem:join-pullback}
    The projection map $\fM_{0,2,\Ran}^{\sst} \to \fM_{0,2}^{\sst}$ is an associative algebra map, and for each $\ell > 0$ we have a Cartesian square \[\begin{tikzcd}
	{\fM_{0,2,\Ran}^{\sst,\ell}} & {\fM_{0,2,\Ran}^{\sst}} \\
	{\fM_{0,2}^{\sst,\ell}} & {\fM_{0,2}^{\sst}}
	\arrow["{\lor_{\Ran}}", from=1-1, to=1-2]
	\arrow[from=1-1, to=2-1]
	\arrow[from=1-2, to=2-2]
	\arrow["\lor", from=2-1, to=2-2]
    \end{tikzcd}\]
\end{lemma}

\begin{proof}
    The pullback $\fM_{0,2,\Ran}^{\sst} \times_{\fM_{0,2}^{\sst}} \fM_{0,2}^{\sst,\ell}$ classifies $\ell$ two-pointed curves and a finite,  stable, and reduced subset of their join. Thus, we have a map \[\fM_{0,2,\Ran}^{\sst} \times_{\fM_{0,2}^{\sst}} \fM_{0,2}^{\sst,\ell} \to \fM_{0,2,\Ran}^{\sst,\ell}~,\] whose $i$-th component is given by the $i$-th curve, together with the intersection of the finite subset with its image under the join map. This defines an inverse to the natural map \[\fM_{0,2,\Ran}^{\sst,\ell} \to \fM_{0,2,\Ran}^{\sst} \times_{\fM_{0,2}^{\sst}} \fM_{0,2}^{\sst,\ell}~.\]
\end{proof}

\begin{lemma} \label{lem:join-proper}
    The maps $\lor$ and $\lor_{\Ran}$ are proper.
\end{lemma}

\begin{proof}
    Using the valuative criterion, we need to show the following: Given a commutative diagram
    \[\begin{tikzcd}
	    {\Spec K} & {\fM^{\sst,\ell}_{0,2}} \\
	    {\Spec A} & {\fM^{\sst}_{0,2}}
	    \arrow["{\mathring{f}}", from=1-1, to=1-2]
	    \arrow[from=1-1, to=2-1]
	    \arrow["\lor", from=1-2, to=2-2]
	    \arrow[dashed, from=2-1, to=1-2]
	    \arrow["f"', from=2-1, to=2-2]
    \end{tikzcd}\]
    for $K = \operatorname{Frac}A, A$ a valuation ring, there is a unique dotted arrow lifting $f$ and extending $\mathring{f}$. Write $(X,x_0,x_\infty)$ for the $A$-family classified by $f$, and $(\mathring{X}^i,\mathring{x}_0^i,\mathring{x}_{\infty}^i); 1 \leq i \leq \ell$ for the $K$-families classified by $\mathring{f}$, so that 
    \[X_K \simeq \mathring{X}^1 \cup \dotsb \cup \mathring{X}^{\ell}~,\]

    Since the nodal locus is closed, the nodes $\mathring{x}_0^i,\mathring{x}_\infty^i$ extend to $A$-points $x_0^i,x_{\infty}^i$, and we get a decomposition
    \[X \simeq X^1 \cup_{x_\infty^1 \sim x_0^2} \dotsb \cup_{x_\infty^{\ell-1} \sim x_0^{\ell}} X^{\ell}~.\]
    The unique extension to $A$ families is then given by $(X^i,x_0^i,x_\infty^i)$.

    Now by Lemma~\ref{lem:join-pullback}, $\lor_{\Ran}$ is given by the pullback of a proper map, and therefore is proper as well.
\end{proof}

\begin{corollary}
    $\lor_{\Ran}^!$ admits a left adjoint $\lor_{\Ran,!}$.
\end{corollary}

\begin{definition} \label{def:join-mon-str}
    For $\cF \in \IndCoh(\fM_{0,2,\Ran}^{\sst})$, write \[\star_{i=1}^{\ell}\cF_i \coloneqq \lor_{\Ran,!}\boxtimes_{i=1}^{\ell}\cF_i\]
    Let $\IndCoh^{\star}(\fM_{0,2,\Ran}^{\sst})$ be the resulting monoidal category.
\end{definition}

\begin{definition}
    Let $\operatorname{JoinAlg}_{0,2}^{\sst}$ be the full subcategory of $\operatorname{Assoc}(\IndCoh^{\star}(\fM_{0,2,\Ran}^{\sst}))$ spanned by associative algebras $A$ such that the maps \[A^{\boxtimes \ell} \to \lor_{\Ran}^!A\] adjoint to the multiplication maps are isomorphisms for each $\ell > 0$. We call such an object a \textit{join algebra}.
\end{definition}

The main result of this section is:
\begin{proposition}
    For a universal factorization algebra $\cA$, $\cA_{0,2}^{\epsilon,\sst}$ of Definition~\ref{def:univ-to-Mgn} admits a canonical structure of a join algebra.
\end{proposition}

\begin{proof}
    Note that the join product maps to the disjoint union map on $\MDisk_{\Ran}$, namely, we have commutative diagrams 
    \[\begin{tikzcd}
	{\fM_{0,2,\Ran}^{\sst,\ell}} & {\fM_{0,2,\Ran}^{\sst}} \\
	{\MDisk_{\Ran}^{\ell}} & {\MDisk_{\Ran}}
	\arrow["{\lor_{\Ran}}", from=1-1, to=1-2]
	\arrow["{\pi^{\ell}_{\dR}}"', from=1-1, to=2-1]
	\arrow["{\pi_{\dR}}", from=1-2, to=2-2]
    \arrow["\sqcup", from=2-1, to=2-2]
    \end{tikzcd}\] with $\pi_{\dR} = \pi_{\dR}^{\fM_{0,2,\Ran}^{\sst}/\fM_{0,2}^{\sst}}$ in the notation of Definition~\ref{def:pi_Ran^XS}. For a universal factorization algebra $\cA$, we have associative identifications \[\cA^{\boxtimes \ell} \iso \sqcup^!\cA~,\] and thus by pullback, associative isomorphisms \[\cA_{0,2}^{\epsilon,\sst,\boxtimes \ell} \iso \lor_{\Ran}^!\cA_{0,2}^{\epsilon,\sst}~.\]
\end{proof}

Now note that the join product of two stable configurations is again stable, and in particular, the join product restricts to an associative multiplication on $\bcM_{0,2,\Ran}$:
\begin{corollary}
    The join product defines an associative multiplication on the category $\IndCoh(\bcM_{0,2,\Ran})$, which we denote by $\IndCoh^{\star}(\bcM_{0,2,\Ran})$. For a universal factorization algebra $\cA$, the vacuum module $\Vac(\cA)_{0,2}^{\epsilon}$ has a structure of an associative algebra in $\IndCoh^{\star}(\bcM_{0,2,\Ran})$.
\end{corollary}

For general $(g,n)$, the join product acts on each marked point:
\begin{definition} \label{def:join-action-Mgn}
    For $1 \leq i \leq n$, let \[\lor_{\Ran}^{r,i} : \bcM_{g,n,\Ran} \times \bcM_{0,2,\Ran} \to \bcM_{g,n,\Ran}\] be the map \[(X,p_1,\dots,p_n,\{x_j^{\red}\}), (Y,q_0,q_\infty,\{y_j^{\red}\}) \mapsto (X \cup_{p_i \sim q_0} Y,p_1,\dots,q_{\infty},\dots,p_n, \{x_j^{\red}\} \cup \{y_j^{\red}\})~.\]
    Similarly, let \[\lor_{\Ran}^{l,i} : \bcM_{0,2,\Ran} \times \bcM_{g,n,\Ran} \to \bcM_{g,n,\Ran}\] be the map \[(Y,q_0,q_\infty,\{y_j^{\red}\}), (X,p_1,\dots,p_n,\{x_j^{\red}\}) \mapsto (Y \cup_{q_{\infty} \sim p_i} X,p_1,\dots,q_{0},\dots,p_n, \{y_j^{\red}\} \cup \{x_j^{\red}\})~.\]
\end{definition}
This defines $n$ commuting right actions, and $n$ commuting left actions. Note, however, that the right and left actions do not commute with each other. The left action can be obtained from the right action by pre-composition with the involution on $\fM_{0,2}^{\sst}$ obtained by switching $x_0$ and $x_\infty$. Similarly to the case of $(g,n) = (0,2)$, we have:
\begin{proposition}
    $\IndCoh(\bcM_{g,n,\Ran})$ admits $n$ commuting right (resp. left) actions of $\IndCoh^{\star}(\bcM_{0,2,\Ran})$. We denote the resulting $\IndCoh^{\star}(\bcM_{0,2,\Ran})$ $n$-module structure by $\IndCoh^{\star,r}(\bcM_{g,n,\Ran})$ (resp. $\IndCoh^{\star,l}(\bcM_{g,n,\Ran})$). 
\end{proposition}

\begin{proposition} \label{prop:fact-to-join}
    For a universal factorization algebra $\cA$, the corresponding vacuum module $\Vac(\cA)_{g,n}^{\epsilon}$ admits $n$ commuting right (resp. left) actions of $\Vac(\cA)_{0,2}^{\epsilon}$, considered as an object of $\IndCoh^{\star,r}(\bcM_{g,n,\Ran})$ (resp. $\IndCoh^{\star,l}(\bcM_{g,n,\Ran})$).
\end{proposition}

\section{Chiral homology for nodal curves}\label{sec:ch-hom}

\subsection{Associative algebra structure}

In this section, we define chiral homology for nodal and punctured curves, and prove a gluing result relating the chiral homology of a nodal curve with that of its normalization.

\begin{definition}
    For $2g + n - 2 > 0$ and a universal factorization algebra $\cA$, define \[\int_{g,n} \cA \coloneqq p_!^0\Vac(\cA)_{g,n}^{0} \in \IndCoh(\bcM_{g,n})~,\] where $\Vac(\cA)_{g,n}^{0}$ is as in Corollary~\ref{cor:vac-0}, and $p^0 : \Ran_{g,n} \to \bcM_{g,n}$ is the projection. For a $k$-algebra $R$ and $f : \Spec R \to \bcM_{g,n}$ classifying an $R$-family $(X,x_1,\dots,x_n)$, define \[\int_{X \backslash \{x_1,\dots,x_n\}} \cA \coloneqq f^!\int_{g,n}\cA \in \IndCoh(\Spec R)~.\]
\end{definition}

Note that $p_!^0 \simeq p_*^0$ is well-defined since $p^0$ is pseudo-proper. Explicitly, $\int_{g,n}\cA$ is given by integrating the vacuum module over all stable configurations in all semistable modifications: \[\int_{g,n}\cA \simeq p^{\epsilon}_!\Vac(\cA)_{g,n}^{\epsilon}\] and \[\int_{X \backslash \{x_1,\dots,x_n\}} \cA \simeq p_{(X,x_1,\dots,x_n),!}^{\epsilon}\Vac(\cA)_{(X,x_1,\dots,x_n)}^{\epsilon}~,\] where $\Vac(\cA)_{(X,x_1,\dots,x_n)}^{\epsilon}$ is as in Definition~\ref{def:VacXx}, $p^{\epsilon} : \bcM_{g,n,\Ran} \to \bcM_{g,n}$ the projection, and $p^{\epsilon}_{X,x_1,\dots,x_n} : \bcM_{(X,x_1,\dots,x_n),\Ran} \to \Spec R$ its base-change along $f$. 

The restriction of $\Vac(\cA)_{(X,x_1,\dots,x_n)}^{0}$ to $\mathring{X} \coloneqq X^{\sm} \backslash \{x_1,\dots,x_n\}$ is just the usual vacuum module, while its fiber over $x_i$ has a structure of a chiral module $\fZ_{\cA}^+$ supported at a point, which we describe below. In particular, if $X$ is smooth, we can realize $\int_{X \backslash \{x_1,\dots,x_n\}} \cA$ as factorization homology of $\cA$ over the proper curve $X$ with coefficients in chiral module $\fZ_{\cA}^+$ at each of the points $x_1,\dots,x_n$. Its value over a node has a structure of a chiral bimodule supported at a point, isomorphic to the bimodule $\fZ_{\cA}$ which we describe below.

\begin{definition}
    For a universal factorization algebra $\cA$, define \begin{align*}
        \fZ_{\cA}^{+} & \coloneqq \Gamma_c(\bcM_{0,2+1,\Ran}^{\land},\Vac(\cA)_{0,3}^{\epsilon}) \\
        \fZ_{\cA}^{-} & \coloneqq \Gamma_c(\bcM_{0,1+2,\Ran}^{\land},\Vac(\cA)_{0,3}^{\epsilon}) \\
        \fZ_{\cA} & \coloneqq \Gamma_c(\bcM_{0,2+2,\Ran}^{\land,\nd} ,\Vac(\cA)_{0,4}^{\epsilon}) \\
        \tilde{\fZ}_{\cA} & \coloneqq \Gamma_c(\bcM_{0,2+2,\Ran}^{\land},\Vac(\cA)_{0,4}^{\epsilon}) \\
        \fZ_{\cA}^{0} & \coloneqq \Gamma_c(\bcM_{0,2,\Ran},\Vac(\cA)_{0,2}^{\epsilon}) ~,\\
    \end{align*} where the spaces $\bcM_{0,2+1,\Ran}^{\land}$ etc. are as in Definition~\ref{def:M2plus2completed}. For a smooth point $x \in X$, define \[\fZ_{\cA,x} = \imath_x^!\Vac(\cA)_{(X,x)}^0\] where $\imath_x : \{x\} \into \Ran X$ is the inclusion.
\end{definition}

Recall from Definition~\ref{def:join-mon-str} the monoidal category $\IndCoh^{\star}(\bcM_{0,2,\Ran})$.
\begin{lemma} \label{lem:glob-join}
    The functor $\Gamma_c : \IndCoh^{\star}(\bcM_{0,2,\Ran}) \to \Vect$ is monoidal.
\end{lemma}

\begin{proof}
    The monoidal structure $\star$ is given by $\lor_{\Ran,!}\boxtimes$, and in particular commutes with ($!$-)pushforward.
\end{proof}

\begin{corollary}
    $\fZ_{\cA}^0$ has a structure of a non-unital associative algebra.
\end{corollary}

\begin{proof}
    By Proposition~\ref{prop:fact-to-join}, $\Vac(\cA)_{0,2}^{\epsilon}$ is an associative algebra in $\IndCoh^{\star}(\bcM_{0,2,\Ran})$. Therefore by Lemma~\ref{lem:glob-join}, its space of global sections $\fZ_{\cA}^0$ inherits an associative algebra structure as well. 
\end{proof}

\begin{proposition}
    $\fZ_{\cA}^{+}$ (resp. $\fZ_{\cA}^{-}$) has a structure of a chiral module supported at $\infty \in \PP^1$ (resp. $0 \in \PP^1$). More generally, for any smooth pointed curve $(X,x)$, the fiber $\fZ_{\cA,x}$ has a structure of a chiral $\cA$-module supported at $x$. These modules are all isomorphic up to a choice of a local coordinate.
\end{proposition}

\begin{proof}
    Note that $\fZ_{\cA}^+$ is the fiber of $\Vac(\cA)_{0,3}^{0}$ at $\{\infty\} \into \Ran \PP^1 \simeq \Ran_{0,3}$, while $\fZ_{\cA}^{-}$ is its fiber at $\{0\} \into \Ran \PP^1$. By Proposition~\ref{prop:structure_MXxRan_smooth}, all these modules are isomorphic up to a choice of a local coordinate. 
\end{proof}

\begin{remark}
    Since $\bcM_{0,2+1,\Ran}^{\land}$ is a completion of $\bcM_{0,2,\Ran}$, the underlying vector space of $\fZ_{\cA}^{+}$ admits a filtration, coming from the adic filtration, such that $\fZ_{\cA}^0$ is the zeroth filtered space.
\end{remark}

\begin{lemma}
    $\fZ_{\cA}^+$ (resp. $\fZ_{\cA}^-$) admits a right (resp. left) $\fZ_{\cA}^0$-module structure, and we have maps of $\fZ_{\cA}^0$-modules \[\fZ_{\cA}^+ \from \fZ_{\cA}^0 \to \fZ_{\cA}^-~.\] There exists an involution \[\theta : \fZ_{\cA}^{0} \iso \fZ_{\cA}^{0,\op}\] such that \[\fZ_{\cA}^- \simeq \theta^*\fZ_{\cA}^+~.\]
\end{lemma}

\begin{proof}
    We have a right (resp. left) $\IndCoh^{\star}(\bcM_{0,2,\Ran})$-module structure on the category $\IndCoh(\bcM_{0,2+1,\Ran}^{\land})$ (resp. $\IndCoh(\bcM_{0,1+2,\Ran}^{\land})$) given by restriction of the join action on $\IndCoh(\bcM_{0,2+1,\Ran})$ (resp. $\IndCoh(\bcM_{0,1+2,\Ran})$). Furthermore, by Proposition~\ref{prop:fact-to-join}, the $(g,n)$-vacuum module admits a module structure with respect to this action. Finally, the involution is obtained from the involution on $\bcM_{0,2,\Ran}$ given by \[(X,x_0,x_\infty) \mapsto (X,x_\infty,x_0)~.\]  
\end{proof}

\begin{proposition}
    For any smooth pointed curve $(X,x_1,\dots,x_n) \in \cM_{g,n}(k)$, $\int_{X \backslash \{x_1,\dots,x_n\}}\cA$ admits $n$ commuting left $\fZ_{\cA}^0$-module structures, and $n$ commuting right $\fZ_{\cA}^0$-module structures.
\end{proposition}

\begin{proof}
    Follows from Proposition~\ref{prop:fact-to-join} together with Lemma~\ref{lem:glob-join}. 
\end{proof}

\subsection{Gluing formula}

In this section we prove the main results of this paper, namely the expression of chiral homology for nodal curves in terms of chiral homology for their normalizations, with the preimage of the nodes considered as punctures. We start with the case of a nodal curve given by gluing two pointed curves along marked points.

\begin{theorem} \label{thm:disj-gluing}
    Given pointed stable curves $(X,x),(Y,y)$ and a universal factorization algebra $\cA$, there is a canonical isomorphism \[\int_{X \underset{x \sim y}{\cup}Y} \cA \simeq \int_{X \backslash \{x\}} \cA \underset{\fZ_{\cA}^0}{\otimes} \int_{Y \backslash \{y\}} \cA~.\]
\end{theorem}

\begin{proof}
    Note that the join map \[\bcM_{(X,x),\Ran} \times \bcM_{(Y,y),\Ran} \to \bcM_{X\underset{x \sim y}{\cup} Y,\Ran}\] is schematic, proper (similarly to Lemma~\ref{lem:join-proper}), and surjective, and its (semisimplicial) \v{C}ech nerve is precisely the (non-unital) bar complex $\bcM_{(X,x),\Ran} \times \bcM_{0,2,\Ran}^{\bullet} \times \bcM_{(Y,y),\Ran}$ computing the tensor product $\bcM_{(X,x),\Ran} \underset{\bcM_{0,2,\Ran}}{\otimes} \bcM_{(Y,y),\Ran}$. By proper descent (see~\cite[Proposition~3.3.3]{GR2}), we get \begin{align*}
        \IndCoh(\bcM_{X\underset{x \sim y}{\cup} Y,\Ran}) & \simeq \lim_{\Delta^{\inj}, (-)^!}\IndCoh(\bcM_{(X,x),\Ran} \times \bcM_{0,2,\Ran}^{\bullet} \times \bcM_{(Y,y),\Ran}) \\
        & \simeq \underset{\Delta^{\inj,\op},(-)_!}{\colim}\IndCoh(\bcM_{(X,x),\Ran} \times \bcM_{0,2,\Ran}^{\bullet} \times \bcM_{(Y,y),\Ran})~.
    \end{align*} Therefore, an object $\cF \in \IndCoh(\bcM_{X\underset{x \sim y}{\cup} Y,\Ran})$ may be written as \[\cF \simeq \colim_{\Delta^{\inj,\op}} \lor_{\Ran,!}^{\bullet}\lor_{\Ran}^{\bullet,!}\cF\] Since $\Gamma_c$ commutes with colimits, we get \[\Gamma_c(\bcM_{X\underset{x \sim y}{\cup} Y,\Ran}, \cF) \simeq \underset{\Delta^{\inj,\op}}{\colim} \Gamma_c(\bcM_{(X,x),\Ran} \times \bcM_{0,2,\Ran}^{\bullet} \times \bcM_{(Y,y),\Ran}, \lor_{\Ran,!}^{\bullet}\lor_{\Ran}^{\bullet,!}\cF)~.\] Taking $\cF = \Vac(\cA)_{X \underset{x \sim y}{\cup}Y}^0$, we get precisely the bar complex computing the tensor product $\int_{X \backslash \{x\}} \cA \underset{\fZ_{\cA}^0}{\otimes} \int_{Y \backslash \{y\}} \cA$.
\end{proof}

\begin{remark}
    As mentioned above, if we fix formal coordinates at $x,y$ and assume $X,Y$ are smooth, the RHS may be interpreted in terms of usual chiral homology for proper curves with coefficients, namely \[\int_{X \backslash \{x\}} \cA \underset{\fZ_{\cA}^0}{\otimes} \int_{Y \backslash \{y\}} \cA \simeq \int_{(X,x)} (\cA, \fZ_{\cA}^+) \underset{\fZ_{\cA}^0}{\otimes} \int_{(Y,y)} (\cA,\fZ_{\cA}^-)~.\]
\end{remark}

We have a similar result for self-gluing:
\begin{theorem} \label{thm:self-gluing}
    For a two-pointed stable curve $(X,x_1,x_2) \in \bcM_{g,2}(k)$ and a universal factorization algebra $\cA$, there is a canonical isomorphism \[\int_{X \underset{\{x_1,x_2\}}{\cup} \{\pt\}} \cA \simeq \int_{(S^1,*)}(\fZ_{\cA}^0, \int_{X \backslash \{x_1,x_2\}}\cA)~.\] Here $\int_{(S^1,*)}(\fZ_{\cA}^0,-)$ is the functor of Hochschild homology of the associative algebra $\fZ_{\cA}^0$ with coefficients in a bimodule. 
\end{theorem}

\begin{proof}
    Similar to the proof of Theorem~\ref{thm:disj-gluing}, we have a proper surjective map \[\bcM_{(X,x_1,x_2),\Ran} \to \bcM_{X \underset{\{x_1,x_2\}}{\cup} \{\pt\},\Ran}\] given by gluing $x_1$ to $x_2$, and its \v{C}ech nerve is precisely the bar complex computing $\bcM_{(X,x_1,x_2),\Ran} \otimes_{\bcM_{0,2,\Ran} \times \bcM_{0,2,\Ran}^{\op}} \bcM_{0,2,\Ran}$. The induced complex on global sections is the bar complex computing \[\int_{X \backslash \{x_1,x_2\}}\cA \underset{\fZ_{\cA}^0 \otimes \fZ_{\cA}^{0,\op}}{\otimes} \fZ_{\cA}^0 = \int_{(S^1,*)}(\fZ_{\cA}^0, \int_{X \backslash \{x_1,x_2\}}\cA)~.\]
\end{proof}

Alternatively, one can deduce the gluing formulas from a gluing formula for the chiral modules $\fZ_{\cA}^{\pm}$, namely we have the following result, which is proven similarly to Theorem~\ref{thm:disj-gluing}.

\begin{lemma}
    The join map \[\bcM_{0,2+1,\Ran}^{\land} \times \bcM_{0,1+2,\Ran}^{\land} \to \bcM_{0,2+2,\Ran}^{\land,\nd}\] induces an isomorphism \[\bcM_{0,2+1,\Ran}^{\land} \underset{\bcM_{0,2,\Ran}}{\otimes} \bcM_{0,1+2,\Ran}^{\land} \simeq \bcM_{0,2+2,\Ran}^{\land,\nd}~,\] and thus an isomorphism \[\fZ_{\cA}^{+} \underset{\fZ_{\cA}^0}{\otimes} \fZ_{\cA}^{-} \simeq \fZ_{\cA}~.\]
\end{lemma}

\begin{remark}
    In the topological setting, given an $\EE_2$-algebra $\cA$, the factorization homology $\int_{S^1 \times \RR^1} \cA$ admits an associative algebra structure, and we have an isomorphism $\cA\mh\Mod^{\EE_2} \simeq (\int_{S^1 \times \RR^1})\mh\Mod$ (\cite[Proposition~3.16]{Fr}). In addition, for any Riemann surface $\Sigma$ with $n$ boundary components $\partial \Sigma \simeq (S^1)^{\sqcup n}$, the factorization homology $\int_{\Sigma} \cA$ admits $n$ commuting left $\int_{S^1 \times \RR^1} \cA$-module structures, and given $\Sigma_1,\Sigma_2$ with $\partial \Sigma_i \simeq S^1$, one has a gluing formula (\cite[Lemma~3.18]{AF}) \[\int_{\Sigma_1 \underset{S^1 \times \RR^1}{\otimes} \Sigma_2} \cA \simeq \int_{\Sigma_1} \cA \underset{\int_{S^1 \times \RR^1} \cA}{\otimes} \int_{\Sigma_2} \cA~.\] In the algebraic setting, we have two different associative algebras associated to a factorization algebra $\cA$: The first is $\fZ_{\cA}^0$, which, by the above, satisfies an analogous gluing formula. The second is the topological associative algebra $\cA_x^{\operatorname{as}}$ of~\cite[Section~3.6]{BD} (or its vertex algebra counterpart, given by the universal enveloping algebra $U(V)$, see~\cite[Section~4.3]{FBZ}), which satisfies $\cA\mh\ChMod_x \simeq \cA_x^{\operatorname{as}}\mh\Mod$. 
\end{remark}

\printbibliography

\end{document}